\documentclass[reqno]{amsart}
\usepackage{amssymb}
\usepackage{graphicx}
\usepackage{amsmath}
\usepackage{mathbbol}

\usepackage[usenames, dvipsnames]{color}
\usepackage{verbatim}
\usepackage{mathrsfs}
\usepackage{bm}
\usepackage{cite}
\usepackage{color}
\allowdisplaybreaks[4]

\numberwithin{equation}{section}

\newtheorem{theorem}{Theorem}[section]
\newtheorem{corollary}[theorem]{Corollary}
\newtheorem{lemma}[theorem]{Lemma}
\newtheorem{prop}[theorem]{Proposition}

\theoremstyle{definition}
\newtheorem{remark}[theorem]{Remark}

\theoremstyle{definition}

\theoremstyle{definition}

\makeatletter
\def\dashint{\operatorname%
{\,\,\text{\bf-}\kern-.98em\DOTSI\intop\ilimits@\!\!}}
\makeatother

\def\\det{\text{\det}}

\def\.5{\frac{1}{2}}

\def\h{\mathcal{H}}

\newcommand{\RN}[1]{%
  \textup{\uppercase\expandafter{\romannumeral#1}}%
}

\renewcommand{\epsilon}{\varepsilon}

\newcounter{marnote}

\begin{document}
	\title[Stress blow-up analysis for Stokes flow in 3D]{Stress blow-up analysis when suspending rigid particles approach boundary in 3D Stokes flow}
	
	\author[H.G. Li]{Haigang Li}
	\address[H.G. Li]{School of Mathematical Sciences, Beijing Normal University, Laboratory of MathematiCs and Complex Systems, Ministry of Education, Beijing 100875, China.}
	\email{hgli@bnu.edu.cn}
	
	\author[L.J. Xu]{Longjuan Xu}
	\address[L.J. Xu]{Academy for Multidisciplinary Studies, Capital Normal University, Beijing 100048, China.}
	\email{longjuanxu@cnu.edu.cn}
	
	\author[P.H. Zhang]{PeiHao Zhang}
	\address[P.H. Zhang]{School of Mathematical Sciences, Beijing Normal University, Beijing 100875, China.}
	\email{202231130004@mail.bnu.edu.cn}


	\date{\today} 
	
	\begin{abstract}
		The stress concentration is a common phenomenon in the study of  fluid-solid model.  In this paper, we investigate the boundary gradient estimates and the second order derivatives estimates for the Stokes flow when the rigid particles approach the boundary of the matrix in dimension three. We classify the effect on the blow-up rates of the stress from the prescribed various boundary data: locally constant case and locally polynomial case. Our results hold for general convex inclusions, including two important cases in practice, spherical inclusions and ellipsoidal inclusions. The blow-up rates of the Cauchy stress in the narrow region are also obtained. We establish the corresponding estimates in higher dimensions greater than three.
	\end{abstract}
	
	\maketitle
	
	\section{Introduction and main results}
	
	\subsection{Background and Problem Formulation}\label{subsec1.1}
	In high-contrast composite materials, the physical fields in the narrow region between two adjacent inclusions or between the inclusions and the matrix boundary, such as the electric fields and the stress fields, may be concentrated caused by this microstructure and become arbitrary large. In the last two decades, there are many important progresses on the study of the stress concentration phenomena for two close-to-touching particles in the contexts of Laplace equation for electrostatic field, Lam\'e system for linear elasticity and Stokes equations for viscous flow. In this paper, we investigate the boundary estimates for the gradient of the solutions to the incompressible Stokes 
	flow in a bounded domain $D\subset\mathbb{R}^{d}\, (d\geq3)$ 
	$$\mu\,\Delta {\bf u}=\nabla p,\quad\nabla\cdot {\bf u}=0,$$
	when the rigid particles suspending inside are very close to the matrix boundary $\partial{D}$. The aim is to classify the effect on the blow-up rates of the gradient and the Cauchy stress from various boundary data.  
	
	Based on the geometry of the particle and the domain and the given boundary data, we construct a class of auxiliary functions to capture the main singularity of the stress term and the pressure term. We first derive the optimal blow-up rates of the stress, via establishing the pointwise upper bounds in the neck region and the lower bounds at the midpoint of the shortest line in between the particle and the boundary, and then proceed to estimate  the second order derivatives provided the domain and the boundary data are sufficiently smooth. The advantage of this method is that when we deal with the pressure term we do not need to solve the stream function. The simplified model in 2D was considered in \cite{LXZ}. This is a continuation. However, Several estimates have to be improved to asymptotic expansion and new difficulty in the derivation of the optimality of the blow-up rate is overcome. Our results hold for general convex particles, including the two important and useful cases in practical application: spherical inclusions and ellipsoid inclusions in dimension three.  Our method also can be extended to deal with the problem in higher dimensions greater than three. The corresponding upper bounds are presented as well. 
	
	There is a long history on the study of the Stokes flow in presence of two circular cylinders. More than one hundred years ago, Jeffrey developed in \cite{Jef1} a separable solution method based on bipolar coordinates and then analyzed in \cite{Jef2} the flow generated by two well-separated rotating circular cylinders. See Wannier \cite{Wann}, Hillairet \cite{H} and the references therein for the related works. However, it is difficult to analyze the singular behavior of the solution when the cylinders are close-to-touching, due to the high complexity of the solution. Recently, Ammari, Kang, Kim, and Yu \cite{AKKY} investigated the stress concentration in the two-dimensional Stokes flow, when inclusions are the two-dimensional cross sections of circular cylinders of the same radii, by using the bipolar coordinates to construct two vector-valued functions to derive an asymptotic representation formula for the stress. The blow-up rate of $|\nabla {\bf u}|$ is proved to be $\varepsilon^{-1/2}$, where $\varepsilon$ is the distance between two cylinders. Laterly, the first two authors \cite{LX} studied the general convex inclusions in dimensions two and three and proved that the optimal blow-up rates of the stress are, respectively, $\varepsilon^{-1/2}$ in dimension two and $(\varepsilon|\ln\varepsilon|)^{-1}$ in dimension three, by adapting the iteration method developed in the study of linear elasticity problem  \cite{LLBY,BLL} to establish pointwise upper bounds and lower bounds at the narrowest place between two inclusions. The estimates in \cite{AKKY,LX} can be regarded as interior estimates. However, due to the effect directly from the boundary data, the solutions of  the Stokes flow may become more irregular near the boundary. To clarify such singularity from the geometry of boundary and given boundary data,  we proceed in this paper to study the stress concentration in the Stokes flow when particles approach the boundary of the matrix in three dimension and higher dimensions.  There are some closely related works  on the expansions of the energy norm, see G\'erard-Varet and Hillairet \cite{Hill6},  Hillairet and Kela\"{i} \cite{Hill1}, and \cite{DM,DMC} for more.
	
	As mentioned at the beginning, the concentration phenomenon also occurs in the electrostatic field and elastic stress composite materials. Since the well-known numerical work of Babu\u{ s}ka, Andersson, Smith and Levin \cite{Bab}, there have been many important progresses in this field, see Bonnetier and Vogelious \cite{BV}, Li and Vogelious \cite{LV} and Li and Nirenberg \cite{LN}. It is of vital importance in mathematics to consider the extreme cases when the parameter in the inclusions degenerated to infinite, in order to investigate the concentration degree depending on the small interparticle distance. In the electrostatics case, when the particles are perfectly conductive, the blow-up rate of the electric field, represented by the gradient of the solution, is proved to be of order $\varepsilon^{-1/2}$ in dimension two and $(\varepsilon|\ln\varepsilon|)^{-1}$ in three  dimensions, see the work by Ammari, Kang, and Lim \cite{AKL}, Bao, Li, and Yin \cite{bly1}. For more literature related to the electric field concentration, we refer the reader to the introduction of \cite{LX} and  \cite{AKL3, AKLLZ,bly2, bt, kly, kly2, LY, Yun,ABTV,BC,BT,BV,DongLi,DongZhang,KLeY},  for instance (this is far from a complete list). For the insulated conductivity problem, see \cite{DLY1,DLY2,bly2,LiY,W,Yun2}. In the context of linear elasticity, when the Lam\'e parameters of  inclusions degenerate to infinity, the optimal blow-up rate of the gradient is proved to be $\varepsilon^{-1/2}$ in dimension two \cite{BLL,KY}, and $(\varepsilon|\ln\varepsilon|)^{-1}$ in dimension three \cite{BLL2,Li2021}. Especially, for the boundary estimates when particles are close to the boundary, see Bao, Ju and Li \cite{BJL}, and Li and Zhao \cite{LZ}. We also refer the reader to \cite{CS1,CS2,gn} for the corresponding nonlinear problem, and \cite{Lu} for the asymptotic behavior of fluid flows in domains perforated with a large number of tiny holes.
	
	Before we state our main results, we first fix our domain and notations. In order to establish the estimate of the gradient and that of the second-order partial derivatives, in this paper we assume that the domain is of class $C^3$ and the boundary data $\boldsymbol\varphi\in C^{2,\alpha}(\partial D;\mathbb R^d)$, because we will use the $W^{2,\infty}$-estimates in \cite[Proposition 3.6]{LX}, see Step 3 in the proof of Proposition \ref{propu11} for the details. Suppose $D\subset\mathbb{R}^{d}$ is a bounded open set with $C^{3}$ boundary, and $D_1^0$ is a $C^{3}$ strictly convex subset of $D$, touching $\partial D$, say, at the origin, and the $x_{1},\dots,x_{d-1}$-axes being in their common tangent plane, after a translation and rotation of coordinates if necessary. We use $x'=(x_{1},\dots,x_{d-1})$. By a translation, set $ D_1^{\varepsilon}:=D_{1}^{0}+(0',\varepsilon)$, for small constant $\varepsilon>0$. 
	For the sake of simplicity, we drop the superscript $\varepsilon$ and denote
	\begin{equation*}
	D_{1}:=D_{1}^{\varepsilon}\, \quad\text{and}\quad  \Omega:=D\setminus\overline{D}_1.
	\end{equation*}
	Then $\overline{D}_{1}\subset D$, $\varepsilon:=\mbox{dist}(D_{1},\partial D)$. 
	We assume that the $C^{3}$ norms of $\partial D$ and $\partial D_1$ are bounded by a constant independent of $\varepsilon$. 
	
	Let
	$$\Psi:=\Big\{{\boldsymbol\psi}\in C^{1}(\mathbb{R}^{d};\mathbb{R}^{d})~|~e({\boldsymbol\psi}):=\frac{1}{2}(\nabla{\boldsymbol\psi}+(\nabla{\boldsymbol\psi})^{\mathrm{T}})=0\Big\},$$
	be the linear space of rigid displacements in $\mathbb{R}^{d}$. It is well known that
	$$\{\boldsymbol{e}_{i},~x_{k}\boldsymbol{e}_{j}-x_{j}\boldsymbol{e}_{k}~|~1\leq\,i\leq\,d,~1\leq\,j<k\leq\,d\}$$
	is a basis of $\Psi$, where $e_{1},\dots,e_{d}$ is the canonical basis of $\mathbb{R}^{d}$. Consider the following Dirichlet problem, describing one rigid particle suspending in the Stokes flow,
	\begin{align}\label{sto}
	\begin{cases}
	\mu \Delta {\bf u}=\nabla p,\quad\quad~~\nabla\cdot {\bf u}=0,&\hbox{in}~~\Omega:=D\setminus\overline{D}_{1},\\
	{\bf u}|_{+}={\bf u}|_{-},&\hbox{on}~~\partial{D_{1}},\\
	e({\bf u})=0, &\hbox{in}~~D_{1},\\
	\int_{\partial{D}_{1}}\frac{\partial {\bf u}}{\partial \nu}\Big|_{+}\cdot{\boldsymbol\psi}_{\alpha}-\int_{\partial{D}_{1}}p\,{\boldsymbol\psi}_{\alpha}\cdot
	\nu=0, &\alpha=1,2,\dots,\frac{d(d+1)}{2},
	\\
	{\bf u}={\boldsymbol\varphi}, &\hbox{on}~~\partial{D},
	\end{cases}
	\end{align}
	where $\mu>0$, ${\boldsymbol\psi}_{\alpha}\in\Psi$, 
	$\frac{\partial {\bf u}}{\partial \nu}\big|_{+}:=\mu(\nabla {\bf u}+(\nabla {\bf u})^{\mathrm{T}})\nu$, $\nu$ is the unit outward normal vector of $D_{1}$, and the subscript $+$ indicates the limit from inside the domain. Furthermore, by using the Gauss theorem and $\nabla\cdot{\bf u}=0$ in $\Omega$, we know the prescribed velocity field ${\boldsymbol\varphi}$ satisfies the compatibility condition:
	\begin{equation}\label{compatibility}
	\int_{\partial{D}}{\boldsymbol\varphi}\cdot \nu\,=0,
	\end{equation}
	which ensures the existence and uniqueness of the solution. See \cite{ADN,Solo,GaldiBook}.
	
	In order to formulate our results precisely,  we assume that near the origin the boundaries $\partial D_1$ and $\partial D$ are expressed by  the graphs of two functions, respectively. Namely, for some $0<R<1$,
	\begin{equation*}
	x_d=\varepsilon+h_1(x')\quad\text{and}\quad x_d=h(x'),\quad \text{for}~ |x'|\leq 2R,
	\end{equation*}
	where $h_1$ and $h$ are $C^{3}$ functions, satisfying 
	\begin{align*}
	&h(x') <\varepsilon+h_{1}(x'),\quad\mbox{for}~~ |x'|\leq 2R,\\
	h_{1}(0)&=h(0)=0,\quad \nabla_{x'} h_{1}(0)=\nabla_{x'}h(0)=0,\\
	h_{1}(x')=\kappa_1|x'|^{2}+&O(|x'|^{3}),\quad h(x')=\kappa|x'|^{2}+O(|x'|^{3}),\quad\mbox{for}~~|x'|<2R,
	\end{align*}
	where we assume that the constants $\kappa_1>\kappa$. For $0\leq r\leq 2R$, we denote the neck region
	\begin{equation}\label{narrow-region}
	\Omega_r:=\left\{(x',x_{d})\in \Omega:~ h(x')<x_{d}<\varepsilon+h_1(x'),~|x'|<r\right\}.
	\end{equation} 
	Throughout this paper, $C$ denotes a {\em{universal constant}}, which means the value of $C$ may vary from line to line, depends only on $d,\mu,\kappa_1,\kappa$, and the upper bounds of the $C^{3}$ norm of $\partial{D}$, $\partial{D}_{1}$, but independent of $\varepsilon$.
	
	\subsection{Upper Bounds of $|\nabla{\bf u}|$ and $|p|$ in 3D}\label{secmain}
	Let  $({\bf u},p)$ be a pair of solution to \eqref{sto}.  We introduce the Cauchy stress tensor
	\begin{equation}\label{defsigma}
	\sigma[{\bf u},p]=2\mu e({\bf u})-p\mathbb{I},
	\end{equation}
	where $\mathbb{I}$ is the identity matrix. Then we reformulate \eqref{sto} as
	\begin{align}\label{Stokessys}
	\begin{cases}
	\nabla\cdot\sigma[{\bf u},p]=0,~~\nabla\cdot {\bf u}=0,&\hbox{in}~~\Omega,\\
	{\bf u}|_{+}={\bf u}|_{-},&\hbox{on}~~\partial{D_{1}},\\
	e({\bf u})=0, &\hbox{in}~~D_{1},\\
	\int_{\partial{D}_{1}}\sigma[{\bf u},p]\cdot{\boldsymbol\psi}_{\alpha}
	\nu=0
	,&\alpha=1,\dots,6,\\
	{\bf u}={\boldsymbol\varphi}, &\hbox{on}~~\partial{D}.
	\end{cases}
	\end{align}
	
	As in \cite{BJL,LXZ}, by the third line in \eqref {Stokessys}, we decompose $\bf u$ in $D_{1}$ as
	\begin{equation*}
	{\bf u}=\sum_{\alpha=1}^{6}C^{\alpha}{\boldsymbol\psi}_{\alpha},\quad\mbox{where}~{\boldsymbol\psi}_{\alpha}\in\Psi,\quad\mbox{in}~{D}_1,
	\end{equation*} 
	where $C^\alpha$, $\alpha=1,2,\dots,6$, are some free constants to be determined by the fourth line in \eqref {Stokessys} later. By virtue of the continuity of the transmission condition on $\partial{D}_{1}$, then $({\bf u},p)$ in $\Omega$ can be split as follows:
	\begin{align}\label{udecom}
	{\bf u}(x)&=\sum_{\alpha=1}^{6}C^{\alpha}{\bf u}_{\alpha}(x)+{\bf u}_{0}(x)\quad\mbox{and}\quad p(x)=\sum_{\alpha=1}^{6}C^{\alpha}p_{\alpha}(x)+p_{0}(x),
	\end{align}
	where ${\bf u}_{\alpha},{\bf u}_{0}\in{C}^{2}(\Omega;\mathbb R^3),~p_{\alpha}, p_0\in{C}^{1}(\Omega)$ satisfy, respectively,
	\begin{equation}\label{equ_v1}
	\begin{cases}
	\nabla\cdot\sigma[{\bf u}_\alpha,p_{\alpha}]=0,\quad\nabla\cdot {\bf u}_{\alpha}=0&\mathrm{in}~\Omega,\\
	{\bf u}_{\alpha}={\boldsymbol\psi}_{\alpha}&\mathrm{on}~\partial{D}_{1},\\
	{\bf u}_{\alpha}=0&\mathrm{on}~\partial{D},
	\end{cases}\quad\alpha=1,2,\dots,6,
	\end{equation}
	and
	\begin{equation}\label{equ_v3}
	\begin{cases}
	\nabla\cdot\sigma[{\bf u}_{0},p_0]=0,\quad\nabla\cdot {\bf u}_{0}=0&\mathrm{in}~\Omega,\\
	{\bf u}_{0}=0&\mathrm{on}~\partial{{D}_{1}},\\
	{\bf u}_{0}={\boldsymbol\varphi}&\mathrm{on}~\partial{D}.
	\end{cases}
	\end{equation}
	Thus,
	\begin{equation}\label{equ_v4}
	\nabla{\bf u}=\sum_{\alpha=1}^{6}C^\alpha\nabla {\bf u}_\alpha+\nabla {\bf u}_0\quad\mathrm{in}~\Omega.
	\end{equation}
	
	For ${\bf u}_\alpha$, $\alpha=1,2,\dots,6$, we will construct auxiliary functions ${\bf v}_{\alpha}$ having the same boundary conditions as ${\bf u}_{\alpha}$ and the associated auxiliary functions ${\bar p}_\alpha$, and then apply the iteration approach as in \cite{BLL,BLL2,LX,LXZ} to prove $\nabla{\bf v}_{\alpha}$ being the main singular terms of $\nabla{\bf u}_{\alpha}$, see Proposition \ref{propu11} below. However, in order to investigate the effect from the boundary data ${\boldsymbol\varphi}$, from \eqref{equ_v3}, it is also important to derive the estimate of $|\nabla{\bf u}_0|$. Let
	\begin{equation*}
	{\bf\Phi}:=\left\{\boldsymbol\varphi\in C^{2,\alpha}(\partial D;\mathbb R^3)~|~\boldsymbol\varphi~\mbox{satisfies}~ \eqref{compatibility}\right\},
	\end{equation*}
	and 
	\begin{equation*}
	\Gamma_{2R}=\left\{(x^\prime,x_{3})\in \Omega:~ x_{3}=h(x^\prime),~|x^\prime|<2R\right\}\subset\partial{D}.
	\end{equation*}
	
	Since the gap is narrowest near the origin, the stress may concentrate there. According to the Taylor expansion of ${\boldsymbol\varphi}$ at the
	origin, we consider two kinds of boundary data: being locally constant on $\Gamma_{2R}$ and being locally polynomial on $\Gamma_{2R}$. 
	Specifically, 
	\begin{equation}\label{defphi1}
	{\bf\Phi}_{1}:=\left\{{\boldsymbol\varphi}\in{\bf\Phi}~|~{\boldsymbol\varphi}=\begin{pmatrix}
	x_i^{l_1}\\
	0\\
	0
	\end{pmatrix},~\mbox{or}~{\boldsymbol\varphi}=\begin{pmatrix}
	0\\
	x_i^{l_1}\\
	0
	\end{pmatrix},~i=1,2,~\mbox{on}~\Gamma_{2R}\right\},
	\end{equation}
	\begin{equation}\label{defphi2}
	{\bf\Phi}_{2}:=\left\{\boldsymbol\varphi\in{\bf\Phi}~|~{\boldsymbol\varphi}=
	\begin{pmatrix}
	0\\
	0\\
	x_i^{l_2}
	\end{pmatrix},~i=1,2,~\mbox{on}~\Gamma_{2R}\right\},
	\end{equation}
	and
	\begin{equation}\label{defphi3}
	{\bf\Phi}_{3}:=\left\{\boldsymbol\varphi\in{\bf\Phi}~|~{\boldsymbol\varphi}=
	\begin{pmatrix}
	x_3^{l_3}\\
	0\\
	0
	\end{pmatrix},~\mbox{or}~{\boldsymbol\varphi}=\begin{pmatrix}
	0\\
	x_3^{l_3}\\
	0
	\end{pmatrix},~\mbox{or}~{\boldsymbol\varphi}=\begin{pmatrix}
	0\\
	0\\
	x_3^{l_3}
	\end{pmatrix},~\mbox{on}~\Gamma_{2R}\right\},
	\end{equation}
	where $l_1=0$ and $l_2=0$ corresponds to the constant case, and $l_1,l_2,l_3\in\mathbb N^+$ corresponds to the polynomial case. We have to construct an auxiliary function ${\bf v}_{0}$ with the same boundary condition as ${\bf u}_{0}$ for each $\boldsymbol\varphi\in{\bf\Phi}_{i}$, $i=1,2,3$, to capture the main singularity of $|\nabla{\bf u}_0|$. The constructions become more complicated and several new difficulties need to overcome. See Propositions \ref{propu15}--\ref{propu162}.
	
	To estimate $C^\alpha$ in \eqref{udecom}, $\alpha=1,2,\dots,6$, from the fourth line of \eqref{Stokessys} and \eqref{equ_v4}, we need to solve the linear system  
	\begin{equation}\label{ce}
	\sum\limits_{\alpha=1}^{6}C^{\alpha}a_{\alpha\beta}=Q_{\beta}[{\boldsymbol\varphi}],
	\end{equation}
	where
	\begin{align}\label{aijbj}
	a_{\alpha\beta}=-\int_{\partial D_1}{\boldsymbol\psi}_\beta\cdot\sigma[{\bf u}_{\alpha},p_{\alpha}]\nu,\quad
	Q_{\beta}[{\boldsymbol\varphi}]=
	\int_{\partial D_1}{\boldsymbol\psi}_\beta\cdot\sigma[{\bf u}_{0},p_{0}]\nu.
	\end{align}
	To this end, here we need to derive their explicit asymptotic expansions of $a_{\alpha\beta}$ and $Q_\beta[\boldsymbol\varphi]$, instead of upper and lower bounds estimates as before. See Lemmas \ref{lema113D}--\ref{lemaQb3} for details. 
	
	Denote
	\begin{equation}\label{55}
	\delta(x^{\prime})=\varepsilon+(\kappa_1-\kappa)|x^{\prime}|^2.
	\end{equation}
	Then we have the following main results on the upper bounds. 
	
	\begin{theorem}\label{mainthm0}
		Assume that $D,D_1,\Omega$, and $\varepsilon$ are defined as in Section \ref{subsec1.1}. For  $\boldsymbol\varphi\in{\bf\Phi}_i$, $i=1,2,3$, let ${\bf u}\in H^1(D;\mathbb R^3)\cap C^1(\bar{\Omega};\mathbb R^3)$ and $p\in L^2(D)\cap C^0(\bar{\Omega})$ be the solution to \eqref{Stokessys}. Then for sufficiently small $0<\varepsilon<1$, the following assertions hold:
		
		(i) for $\boldsymbol\varphi\in{\bf\Phi}_1$, 
		\begin{align}\label{estDu}
		|\nabla{{\bf u}}(x)|\leq \frac{C(1+|\ln\varepsilon||x'|)}{|\ln\varepsilon|\delta(x')},~x\in\Omega_R,
		\end{align}
		\begin{align}\label{D2uDp}
		|\nabla^2{{\bf u}}(x)|+|\nabla p(x)|\leq
		\frac{C(1+|\ln\varepsilon|\sqrt{\delta(x')})}{|\ln\varepsilon|\delta^2(x')},~x\in\Omega_R,
		\end{align}
		and
		\begin{align*}
		\inf_{c\in\mathbb{R}}\|p+c\|_{C^0(\bar{\Omega}_{R})}\leq
		\begin{cases}
		\frac{C}{|\ln\varepsilon|\varepsilon^{3/2}},&  l_1=0,~ l_1\geq2,\\
		\frac{C}{\varepsilon^{3/2}},&  l_1=1,
		\end{cases} 
		\end{align*}
		where $C$ is a {\em{universal} constant} and $\delta(x')$ is defined in \eqref{55};
		
		(ii) for $\boldsymbol\varphi\in{\bf\Phi}_2$,  
		\begin{align*}
		|\nabla{{\bf u}}(x)|\leq
		\begin{cases} \frac{C(1+|\ln\varepsilon||x'|)}{|\ln\varepsilon|\delta(x')},&l_2=0,~ l_2\geq2,\\
		\frac{C}{\delta(x')},& l_2= 1,
		\end{cases}~x\in\Omega_R,
		\end{align*}
		\begin{align*}
		|\nabla^2{{\bf u}}(x)|+|\nabla p(x)|\leq
		\begin{cases} 
		\frac{C(1+|\ln\varepsilon|\sqrt{\delta(x')})}{|\ln\varepsilon|\delta^2(x')},&l_2=0,~l_2\geq2,\\
		\frac{C}{\delta^2(x')},&l_2=1,
		\end{cases}~x\in\Omega_R,
		\end{align*}
		and
		\begin{align*}
		\inf_{c\in\mathbb{R}}\|p+c\|_{C^0(\bar{\Omega}_{R})}\leq
		\begin{cases}
		\frac{C}{|\ln\varepsilon|\varepsilon^{3/2}},&  l_2=0,\\
		\frac{C}{\varepsilon^{3/2}},& l_2\geq1;
		\end{cases} 
		\end{align*}

		(iii) for $\boldsymbol\varphi\in{\bf\Phi}_3$, \eqref{estDu} and \eqref{D2uDp} also hold, and 
		\begin{align*}
		\inf_{c\in\mathbb{R}}\|p+c\|_{C^0(\bar{\Omega}_{R})}\leq
		\frac{C}{|\ln\varepsilon|\varepsilon^{3/2}};
		\end{align*}
		and outside the narrow region, we have
		$$\|\nabla^{k_1}{\bf u}\|_{L^\infty(\Omega\setminus\Omega_R)}+\|\nabla^{k_2} p\|_{L^\infty(\Omega\setminus\Omega_R)}\leq C\|\boldsymbol\varphi\|_{C^{2,\alpha}(\partial D)},\quad k_1=1,2,~k_2=0,1.$$
	\end{theorem}  
	
	\begin{remark}
		In particular, we have
		\begin{equation*}
		\|\nabla{{\bf u}}\|_{L^{\infty}(\Omega)}\leq  
		\begin{cases}
		\frac{1}{\varepsilon|\ln\varepsilon|},&\boldsymbol\varphi\in{\bf\Phi}_i,~i=1,3;~\mbox{or}~\boldsymbol\varphi\in{\bf\Phi}_2,~l_2=0,~l_2\geq 2,\\
		\frac{C}{\varepsilon},&\boldsymbol\varphi\in{\bf\Phi}_2,~l_2=1.
		\end{cases}
		\end{equation*}
		Although the above results for $\boldsymbol\varphi\in{\bf\Phi}_1$ and $\boldsymbol\varphi\in{\bf\Phi}_3$ look the same, the construction of the auxiliary functions are totally different, see Section \ref{sec2outline2D}. 
	\end{remark}

	\begin{remark}
		Our proof also works for more general convex particles only with a minor modification. For example, when
		$$h_1(x)=\kappa_1x_1^2+\kappa_2x_2^2,\quad h(x)=\kappa|x'|^2,\quad\kappa_1,\kappa_2>\kappa,\quad\mbox{in}~\Omega_{2R},$$
		then $\delta(x^{\prime})=\varepsilon+(\kappa_1-\kappa)x_1^2+(\kappa_2-\kappa)x_2^2
		$. The details are given in Section \ref{secellipsoid}. The argument in the proof of Theorem \ref{mainthm0}  works well for the case of $m$-convex inclusions, say, $h_1(x)=\kappa_1|x'|^{m}$ and $h(x)=\kappa|x'|^m$ with $m\geq3$, by modifying the auxiliary functions in Section \ref{sec2outline2D} slightly. We leave the details to the interested readers.
	\end{remark}   
	
	\subsection{Lower Bounds of $|\nabla {\bf u}|$ in 3D} 
	Notice that the upper bounds of $|\nabla{\bf u}|$ established in Theorem \ref{mainthm0}  achieve their  maximum at the shortest line $\{|x'|=0\}\cap\Omega$. We shall prove the lower bounds of $|\nabla {\bf u}|$ at the midpoint of the shortest line in order to show the optimality of the blow-up rate. Assume that $({\bf u}_0^{*},p_0^{*})$ and $({\bf u}_\beta^*,p_\beta^*)$, $\beta=1,\dots,6$, satisfy 
	\begin{align*}
	\begin{cases}
	\nabla\cdot\sigma[{\bf u}_0^{*},p_0^{*}]=0,&\hbox{in}\ \Omega^{0},\\
	\nabla\cdot {\bf u}_0^{ *}=0,&\hbox{in}\ \Omega^{0}\\
	{\bf u}_0^{*}=0,&\hbox{on}\ \partial D_{1}^{0},\\
	{\bf u}_0^{*}={\boldsymbol\varphi},&\hbox{on}\ \partial{D},
	\end{cases}\quad
	\begin{cases}
	\nabla\cdot\sigma[{\bf u}_\beta^*,p_\beta^*]=0,&\mathrm{in}~\Omega^{0},\\
	\nabla\cdot {\bf u}_\beta^*=0,&\mathrm{in}~\Omega^{0},\\
	{\bf u}_\beta^*={\boldsymbol\psi}_{\beta},&\mathrm{on}~\partial{D}_{1}^{0}\setminus\{0\},\\
	{\bf u}_\beta^*=0,&\mathrm{on}~\partial{D},
	\end{cases}
	\end{align*} 
	respectively. 
	
	For  the locally constant boundary data ${\boldsymbol\varphi}\in{\bf\Phi}_{1}$ with $l_1=0$, it follows from Lemma \ref{lemaQb1} that $Q_{\beta}[{\boldsymbol\varphi}]$, defined by \eqref{aijbj}, goes to infinity as $\varepsilon\rightarrow0$, $\beta=1$ or $\beta=5$. For this, we define a new functional by
	\begin{align}\label{defQj}
	Q_{1,\beta}[{\boldsymbol\varphi}]:=Q_{\beta}[{\boldsymbol\varphi}]-a_{1\beta}=\int_{\partial D_{1}}{\boldsymbol\psi}_\beta\cdot\sigma[{\bf u}_0+{\bf u}_{1},p_0+p_1]\nu,\quad\beta=1,\dots,6,
	\end{align}
	where $a_{1\beta}$ is defined by \eqref{aijbj}. Then $|Q_{1,\beta}[{\boldsymbol\varphi}]|\leq C$ since ${\bf u}_0+{\bf u}_{1}$ takes the same data on $\partial D_1$ and $\partial D$. Moreover, we shall prove in Lemma \ref{lemmatildeQ} that $Q_{1,\beta}[{\boldsymbol\varphi}]$ converges to  
	\begin{align}\label{defQj*}
	Q^*_{1,\beta}[{\boldsymbol\varphi}]:=\int_{\partial D_{1}^{0}}{\boldsymbol\psi}_\beta\cdot\sigma[{\bf u}_0^{*}+{\bf u}_{1}^*,p_0^{ *}+p_1^{ *}]\nu,\quad\beta=1,\dots,6.
	\end{align} 
	Assume that ${\boldsymbol\varphi}\in  C^{1,\alpha}(\partial D;\mathbb R^3)$ satisfies the following condition:
	\begin{align*}
	({\mathscr{H}}): &~Q_{1,1}^*[{\boldsymbol\varphi}]-(\kappa_1+\kappa) Q_{1,5}^*[{\boldsymbol\varphi}]-\frac{a^*_{14}-(\kappa_1+\kappa)a^*_{54}}{a^*_{44}}Q^*_{1,4}[{\boldsymbol\varphi}]\neq 0\\
	&~\mbox{for~some}~{\boldsymbol\varphi}\in{\bf\Phi}_{1}~\mbox{with~}l_1=0,
	\end{align*}
	where 
	\begin{equation}\label{equ_a*}
	a^*_{\alpha\beta}:=\int_{\partial D_{1}^{0}}{\boldsymbol\psi}_\beta\cdot\sigma[{\bf u}_\alpha^{*},p_\alpha^{ *}]\nu,\quad\alpha,\beta=1,\dots,6.
	\end{equation}
	Then we have the lower bound as follows.
	
	\begin{theorem}\label{main1thm02}
		Assume that $D,D_1,\Omega$, and $\varepsilon$ are defined as in Section \ref{subsec1.1}. Let ${\bf u}\in H^1(D;\mathbb R^3)\cap C^1(\bar{\Omega};\mathbb R^3)$ and $p\in L^2(D)\cap C^0(\bar{\Omega})$ be the solution to \eqref{Stokessys}. 
		If ${\boldsymbol\varphi}$ satisfies $({\mathscr{H}})$, then for sufficiently small $\varepsilon>0$, we have 
		\begin{align*}
		|\nabla{\bf u}(0,0,\varepsilon/2)|\geq\frac{1}{\varepsilon|\ln\varepsilon|}.
		\end{align*}
	\end{theorem}
	
	\begin{remark}
		Similarly, for ${\boldsymbol\varphi}\in{\bf\Phi}_{2}$ with $l_2=0$, we define for $\beta=1,\dots,6$,
		\begin{align*}
		Q_{2,\beta}[{\boldsymbol\varphi}]:=\int_{\partial D_{1}}{\boldsymbol\psi}_\beta\cdot\sigma[{\bf u}_0+{\bf u}_{2},p_0+p_2]\nu,~ Q^*_{2,\beta}[{\boldsymbol\varphi}]:=\int_{\partial D_{1}^{0}}{\boldsymbol\psi}_\beta\cdot\sigma[{\bf u}_0^{*}+{\bf u}_{2}^*,p_0^{ *}+p_2^{ *}]\nu.
		\end{align*}
		While, for the polynomial cases, $l_1\geq 1 $ and $l_2\geq 2$, we define
		\begin{align*}
		Q^*_{\beta}[{\boldsymbol\varphi}]:=\int_{\partial D_{1}^{0}}{\boldsymbol\psi}_\beta\cdot\sigma[{\bf u}_0^{*},p_0^{ *}]\nu,\quad\beta=1,\dots,6.
		\end{align*}
		By a similar proof as that of Lemma \ref{lemmatildeQ}, we have 
		$$Q_{2,\beta}[{\boldsymbol\varphi}]\rightarrow Q^*_{2,\beta}[{\boldsymbol\varphi}]~\mbox{and}~ Q_{\beta}[{\boldsymbol\varphi}]\rightarrow Q^*_{\beta}[{\boldsymbol\varphi}]\quad\mbox{as}~\varepsilon\rightarrow0.$$
		Then if ${\boldsymbol\varphi}\in  C^{1,\alpha}(\partial D;\mathbb R^3)$ satisfies one of the following conditions
		\begin{align*}
		({\mathscr{H}}_2):&~ Q_{2,1}^*[{\boldsymbol\varphi}]-(\kappa_1+\kappa)	 Q_{2,5}^*[{\boldsymbol\varphi}]-\frac{a^*_{14}-(\kappa_1+\kappa)a^*_{54}}{a^*_{44}}Q^*_{2,4}[{\boldsymbol\varphi}]\neq 0\\
		&~\mbox{for~some}~{\boldsymbol\varphi}\in{\bf\Phi}_{2}~\mbox{with~}l_2=0,
		\end{align*}
		or
		\begin{align*}
		({\mathscr{H}}_3):&~ Q_{1}^*[{\boldsymbol\varphi}]-(\kappa_1+\kappa)	 Q_{5}^*[{\boldsymbol\varphi}]-\frac{a^*_{14}-(\kappa_1+\kappa)a^*_{54}}{a^*_{44}}Q^*_{4}[{\boldsymbol\varphi}]\neq 0\qquad\\
		&~\mbox{for~some}~{\boldsymbol\varphi}\in{\bf\Phi}_{j}~\mbox{with~}l_1\geq 1,~l_2\geq 2,~j=1,2,
		\end{align*}
		the result in Theorem \ref{main1thm02} still holds true.
	\end{remark}
	
	\begin{remark}
		From the perspective of practical application, it is more important to investigate the hydrodynamic force and the hydrodynamic torque acting on the rigid particle $D_1$ defined by
		$${\bf F}=-\int_{\partial D_1}\sigma[{\bf u},p]\nu,\quad{\bf T}= \int_{\partial D_1}(x-x_{D_1})\times\sigma[{\bf u},p]\nu,$$ 
		where $\nu$ is the unit outer normal to $\partial D_1$ and $x_{D_1}$ is the center of the mass of $D_1$. Using the fourth line of \eqref{sto}, here we have ${\bf F}=0$. 
		While, as a consequence of Theorem \ref{mainthm0}, for  given boundary data $\boldsymbol\varphi\in{\bf\Phi}_i$, $i=1,2,3$, we have 
		$$|{\bf T}|\leq C,$$ 
		where the constant $C$ is independent of $\varepsilon$. 
	\end{remark}
	
	\subsection{Upper bounds of $|\nabla{\bf u}|$ and $p$ in dimensions $d\geq4$, and upper bounds of Cauchy stress tensor $\sigma[{\bf u},p]$}
	
	As mentioned before, our method works in any dimension. Next we will present the upper bounds of $|\nabla{\bf u}|$ and $p$ in dimensions $d\geq4$ when the boundary data $\boldsymbol\varphi=x_1^l e_1$ on $\Gamma_{2R}$. Other cases can be handled similarly, we leave it to interested readers.
	
	\begin{theorem}\label{thmhigher}
		Assume that $D,D_1,\Omega$, and $\varepsilon$ are defined as in Section \ref{subsec1.1}. For  $\boldsymbol\varphi=x_1^l e_1$,  let ${\bf u}\in H^1(D;\mathbb R^d)\cap C^1(\bar{\Omega};\mathbb R^d)$ and $p\in L^2(D)\cap C^0(\bar{\Omega})$ be the solution to \eqref{Stokessys}. Then for sufficiently small $0<\varepsilon<1$,  we have  
		\begin{align}\label{uppD4}
		|\nabla{{\bf u}}(x)|\leq \frac{C(\varepsilon+|x'|)}{\delta^2(x')},\quad x\in\Omega_R;\quad\inf_{c\in\mathbb{R}}\|p+c\|_{C^0(\bar{\Omega}_{R})}\leq
		\frac{C}{\varepsilon^{2}},
		\end{align}
		and $\|\nabla{\bf u}\|_{L^\infty(\Omega\setminus\Omega_R)}+\|p\|_{L^\infty(\Omega\setminus\Omega_R)}\leq C\|\boldsymbol\varphi\|_{C^{2,\alpha}(\partial D)}$,
		where $C$ is a {\em{universal} constant} and $\delta(x')$ is defined in \eqref{55}.
	\end{theorem}

	Recalling the definition of $\sigma[{\bf u},p]$ in \eqref{defsigma}, as a consequence of Theorems \ref{mainthm0} and \ref{thmhigher},  we have  the estimates for the Cauchy stress tensor $\sigma[{\bf u},p]$ as follows.
	
	\begin{corollary}
		Assume that $D,D_1,\Omega$, and $\varepsilon$ are defined as in Section \ref{subsec1.1}. Let ${\bf u}\in H^1(D;\mathbb R^d)\cap C^1(\bar{\Omega};\mathbb R^d)$ and $p\in L^2(D)\cap C^0(\bar{\Omega})$ be the solution to \eqref{Stokessys} with $d\geq3$. Then for sufficiently small $0<\varepsilon<1$,  we have  the following assertions:
		
		(i) When $d=3$, it holds that, 
		\begin{enumerate}
			\item
			if $\boldsymbol\varphi\in{\bf\Phi}_1$,  then 
			\begin{align*}
			\inf_{c\in\mathbb{R}}\|\sigma[{\bf u},p+c]\|_{C^0(\bar{\Omega}_{R})}\leq 
			\begin{cases}
			\frac{C}{|\ln\varepsilon|\varepsilon^{3/2}},&  l_1=0,~ l_1\geq2,\\
			\frac{C}{\varepsilon^{3/2}},&  l_1=1;
			\end{cases} 
			\end{align*}
			\item if $\boldsymbol\varphi\in{\bf\Phi}_2$,  then 
			\begin{align*}
			\inf_{c\in\mathbb{R}}\|\sigma[{\bf u},p+c]\|_{C^0(\bar{\Omega}_{R})}\leq 
			\begin{cases}
			\frac{C}{|\ln\varepsilon|\varepsilon^{3/2}},&  l_2=0,\\
			\frac{C}{\varepsilon^{3/2}},&  l_2\geq1;
			\end{cases} 
			\end{align*}
			\item  if $\boldsymbol\varphi\in{\bf\Phi}_3$,  then 
			\begin{align*}
			\inf_{c\in\mathbb{R}}\|\sigma[{\bf u},p+c]\|_{C^0(\bar{\Omega}_{R})}\leq 
			\frac{C}{|\ln\varepsilon|\varepsilon^{3/2}}.
			\end{align*}
		\end{enumerate}
		
		(ii) When $d\geq4$, for example, if $\boldsymbol\varphi=x_1^l e_1$ on $\Gamma_{2R}$, we have 
		\begin{align*}
		\inf_{c\in\mathbb{R}}\|\sigma[{\bf u},p+c]\|_{C^0(\bar{\Omega}_{R})}\leq 
		\frac{C}{\varepsilon^{2}}.
		\end{align*}
	\end{corollary}
	
	\begin{remark}
		From the proof of Theorem \ref{mainthm0} in Section \ref{sec5}, one can see that it is quite challenging in dimension three to fix the lower bounds of $p$ and $\sigma[{\bf u},p]$ by using our argument in this paper. For example, if $\boldsymbol\varphi\in{\bf\Phi}_1$ and $l_1=0$ in dimension three, then it follows from \eqref{estpl1=0} that the largest term $\frac{1}{|\ln\varepsilon|\varepsilon^{3/2}}$ comes from $|C^1-1|| p_1(x)-(q_1)_{\Omega_R}|$, $|C^2||p_2-(q_2)_{\Omega_R}|$, and $\sum_{\alpha=5}^{6}|C^\alpha||p_\alpha-(q_\alpha)_{\Omega_R}|$. However, it is difficult to determine the lower bounds of these free constants $C^1-1$, $C^2$, $C^5$, and $C^6$. We leave this problem to interested readers and look forward to new idea and technique to prove the lower bounds and the optimal blow-up rates of the Cauchy stress tensor.
	\end{remark}
	
	Our paper is organized as follows. In Section \ref{sec2outline2D}, we construct the auxiliary functions $\bf v_\alpha$ and the associated $p_\alpha$ in dimension three, present the required properties in Lemma \ref{lemauxi} in order to apply the framework for gradient estimates established in \cite{LX}, and obtain the estimates of $|\nabla \bf u_\alpha|$ and $p_\alpha$ in Proposition \ref{propu11}.  Several kinds of prescribed boundary data ${\boldsymbol\varphi}$ are considered and the similar estimates for $|\nabla {\bf u}_0|$ and $p_0$ are also obtained in Section \ref{subsec22}. By virtue of these precise estimates and asymptotic fomulas for $(\nabla{\bf u}_\alpha,p_\alpha)$ and $(\nabla{\bf u}_0,p_0)$, Section \ref{sec4} is devoted to the estimates of $C^\alpha$. Then in Section \ref{sec5} we prove Theorem \ref{mainthm0} for the upper bounds, and Theorem \ref{main1thm02}  for the lower bounds in 3D.  We study the ellipsoid inclusions case, with two different principle curvatures near the origin in Section \ref{secellipsoid}. Finally, the proof of Theorem \ref{thmhigher} for the upper bounds in dimensions greater than three is given in Section \ref{prfhigher4}. 
	
	\section{Estimates of $(\nabla{\bf u}_\alpha,p_\alpha)$ in 3D}\label{sec2outline2D}
	
	In this section, we prove the estimates for $(\nabla{\bf u}_\alpha,p_\alpha)$, $\alpha=1,2,\dots,6$. By the framework for the gradient estimates developed in \cite{LX}, we reduce the estimates for $(\nabla{\bf u}_\alpha,p_\alpha)$ to the constructions of the appropriate auxiliary functions $(\nabla{\bf v}_\alpha,\bar{p}_\alpha)$, which will be proved to be the main singular terms for each part. To express our idea clearly and simplify the calculation, without loss of generality, we assume that $h_{1}$ and $h$ are quadratic, say, $h_1(x^{\prime})=\kappa_1|x^{\prime}|^2$, $h_2(x^{\prime})=\kappa |x^{\prime}|^2$ for $|x^{\prime}|\leq 2R$, where $\kappa_1$, $\kappa$ are two constants and  $\kappa_0=\kappa_1-\kappa>0$. 
	
	We first define a scalar Keller-type auxiliary function $k(x)\in C^{2}(\mathbb{R}^{d})$,
	satisfying $k(x)=\frac{1}{2}$ on $\partial D_{1}$, $k(x)=-\frac{1}{2}$ on $\partial D$, especially,
	\begin{equation}\label{def_vx}
	k(x)=\frac{x_{3}-h(x^{\prime})}{\delta(x^{\prime})}-\frac{1}{2}\quad\hbox{in}\ \Omega_{2R},
	\end{equation}
	and $
	\|k(x)\|_{C^{2}(\mathbb{R}^{d}\backslash\Omega_{R})}\leq C$.
	Clearly,
	\begin{align*}
	\partial_{x_j}k(x)=-\frac{(\kappa_1+\kappa)x_j}{\delta(x^{\prime})}-\frac{2(\kappa_1-\kappa)x_j}{\delta(x^{\prime})}k(x),~\,j=1,2,\quad
	\partial_{x_3}k(x)
	=\frac{1}{\delta(x')},\quad\hbox{in}\ \Omega_{2R}.
	\end{align*}
	
	For $\alpha=1,2$, for convenience, we define 
	\begin{align}\label{defFGH}
	\begin{split}
	F_\alpha(x)&:=-\frac{12}{5}\frac{(\kappa_1+\kappa)x_\alpha^2}{\delta(x^\prime)}+\frac{3}{5}\frac{\kappa_1+\kappa}{\kappa_1-\kappa},\quad
	G(x):=-\frac{12}{5}\frac{(\kappa_1+\kappa)x_1x_2}{\delta(x^\prime)},\\
	H_\alpha(x)&:=(\kappa_1-\kappa)x_\alpha+2(\kappa_1+\kappa)x_\alpha k(x).
	\end{split}
	\end{align}    
	Now we construct ${\bf v}_{\alpha}\in C^{2}(\Omega;\mathbb R^3)$, such that ${\bf v}_{\alpha}={\bf u}_{\alpha}={\boldsymbol\psi}_{\alpha}$ on $\partial{D}_{1}$ and ${\bf v}_{\alpha}={\bf u}_{\alpha}=0$ on $\partial{D}$, especially,
	\begin{align}\label{v1alpha}
	{\bf v}_{\alpha}=\boldsymbol\psi_{\alpha}\big(k(x)+\frac{1}{2}\big)+{\bf E}_\alpha(x)
	\Big(k^2(x)-\frac{1}{4}\Big),\quad\alpha=1,2,\quad \mbox{in}~\Omega_{2R},
	\end{align}
	and $\|{\bf v}_{\alpha}\|_{C^{2}(\Omega\setminus\Omega_{R})}\leq\,C$, where 
	\begin{equation*}
	{\bf E}_1(x)=\Big(F_1(x),G(x),H_1(x)-\delta(x^\prime)\partial_{x_1}k(x)F_1(x)-\delta(x^\prime)\partial_{x_2}k(x)G(x)\Big)^{\mathrm T},
	\end{equation*}
	and 
	\begin{equation*}
	{\bf E}_2(x)=\Big(G(x),F_2(x),H_2(x)-\delta(x^\prime)\partial_{x_1}k(x)G(x)-\delta(x^\prime)\partial_{x_2}k(x)F_2(x)\Big)^{\mathrm T}.
	\end{equation*}
	To cancel out the biggest term in $\Delta{\bf v}_{\alpha}$, $\alpha=1,2$, and to make the technique for gradient estimates developed in \cite{LX} work well, the associated pressure $\bar{p}_\alpha\in C^{1}(\Omega)$ are chosen to be
	\begin{equation}\label{defp113D}
	\bar{p}_\alpha=\frac{6\mu}{5}\frac{\kappa_1+\kappa}{\kappa_1-\kappa}\frac{ x_{\alpha}}{\delta^2(x^\prime)}+\mu\partial_{x_3} {\bf v}_{\alpha}^{(3)},\quad\alpha=1,2,\quad\mbox{in}~\Omega_{2R},
	\end{equation}
	and  $\|\bar{p}_\alpha\|_{C^{1}(\Omega\setminus\Omega_{R})}\leq C$.
	
	For $\alpha=3$, similarly, we seek ${\bf v}_{3}\in C^{2}(\Omega;\mathbb R^3)$ satisfying,
	\begin{align}\label{v2alpha}
	{\bf v}_{3}=\boldsymbol\psi_{3}\big(k(x)+\frac{1}{2}\big)+{\bf E}_3(x)
	\Big(k^2(x)-\frac{1}{4}\Big),\quad \mbox{in}~\Omega_{2R},
	\end{align}
	and $\|{\bf v}_{3}\|_{C^{2}(\Omega\setminus\Omega_{R})}\leq\,C$, where
	\begin{align*}
	{\bf E}_3(x)=\Big(\frac{3x_1}{\delta(x^\prime)},\frac{3x_2}{\delta(x^\prime)},-2k(x)+\frac{3x_1}{\delta(x^\prime)}\tilde H_1(x)+\frac{3x_2}{\delta(x^\prime)}\tilde H_2(x)\Big)^{\mathrm T},
	\end{align*}
	where 
	\begin{align}\label{defH}
	\tilde H_\alpha=(\kappa_1+\kappa)x_\alpha+2(\kappa_1-\kappa)x_\alpha k(x),\quad\alpha=1,2,
	\end{align}
	and the associated $\bar{p}_3\in C^{1}(\Omega)$ satisfying, in $\Omega_{2R}$,
	\begin{equation*}
	\bar{p}_3=-\frac{3\mu}{2(\kappa_1-\kappa)\delta^2(x^\prime)}+\mu\partial_{x_3} {\bf v}_3 ^{(3)}.
	\end{equation*} 
	
	For $\alpha=4$, the construction of ${\bf v}_{4}\in C^{2}(\Omega;\mathbb R^3)$ is easy and as follows:
	\begin{align}\label{v4alpha}
	{\bf v}_{4}=\boldsymbol\psi_{4}\big(k(x)+\frac{1}{2}\big)\quad\mbox{in}~\Omega_{2R},
	\end{align}
	and here we can directly take $\bar{p}_4=0$.
	
	For $\alpha=5,6$, we define 
	\begin{align*}
	F_\alpha(x)= \frac{12}{5}\frac{x_{\alpha-4}^2}{\delta(x^\prime)}-2k(x)x_3-\frac{3x_3^2}{\delta(x^\prime)}+\frac{3}{5(\kappa_1-\kappa)},\quad \tilde G(x)=-\frac{12}{5}\frac{x_1x_2}{\delta(x^\prime)}.
	\end{align*}
	The auxiliary function ${\bf v}_{\alpha}\in C^{2}(\Omega;\mathbb R^3)$ satisfies
	\begin{align}\label{v5alpha}
	{\bf v}_{\alpha}=\boldsymbol\psi_{\alpha}\big(k(x)+\frac{1}{2}\big)+{\bf E}_\alpha(x)
	\Big(k^2(x)-\frac{1}{4}\Big),\quad \alpha=5,6,\quad \mbox{in}~\Omega_{2R},
	\end{align}
	and $ \|{\bf v}_{\alpha}\|_{C^{2}(\Omega\setminus\Omega_{R})}\leq\,C$, where 
	\begin{align*}
	{\bf E}_5(x)=\left(F_5(x),\tilde G(x),2k(x)x_1+\tilde H_1(x)\Big(F_5(x)+2k(x)x_3\Big)-\delta(x')\partial_{x_2}k(x)\tilde G(x)\right)^{\mathrm T},
	\end{align*}
	and 
	\begin{align*}
	{\bf E}_6(x)=\left(\tilde G(x),F_6(x),2k(x)x_2+\tilde H_2(x)\Big(F_6(x)+2k(x)x_3\Big)-\delta(x')\partial_{x_1}k(x)\tilde G(x)\right)^{\mathrm T},
	\end{align*}
	where $\tilde H_1(x)$ and $\tilde H_2(x)$ are defined in \eqref{defH}. Here the associated pressure $\bar{p}_\alpha\in C^{1}(\Omega)$ in $\Omega_{2R}$, are chosen to be
	\begin{equation*}
	\bar{p}_\alpha=\frac{6\mu}{5(\kappa_1-\kappa)}\frac{ x_{\alpha-4}}{\delta^2(x^\prime)}+\mu\partial_{x_3} {\bf v}_{\alpha}^{(3)},\quad\alpha=5,6,\quad \mbox{in}~\Omega_{2R}.
	\end{equation*}
	
	By direct calculations, these auxiliary functions have the following properties, which is necessary to apply the technique in \cite{LLBY,LX} to estimate $(\nabla{\bf u_\alpha}, p_\alpha)$. 
	
	\begin{lemma}\label{lemauxi}
		Let ${\bf v}_\alpha$ and $\bar p_\alpha$ are defined as above, $\alpha=1,\dots,6$. Then in $\Omega_{2R}$, the following assertions hold.
		
		(1) For $\alpha=1,\dots,6$, $\nabla\cdot{\bf v}_\alpha=0$;
		
		(2) 
		\begin{align*}
		|\nabla{\bf v}_{\alpha}(x)|\leq
		\begin{cases}
		\frac{C}{\delta(x')},& \alpha=1,2,5,6,\\
		C\big(\frac{1}{\delta(x^\prime)} +\frac{|x^\prime|}{\delta^2(x^\prime)}\big),& \alpha=3,\\
		C\big(\frac{|x^\prime|}{\delta(x^\prime)} +1\big),& \alpha=4;
		\end{cases}
		\end{align*}
		and
		
		(3) 
		\begin{align*}
		|\mu\Delta{\bf v}_{\alpha}-\nabla\bar{p}_\alpha|\leq 
		\begin{cases}
		\frac{C}{\delta(x^\prime)},& \alpha=1,2,5,6,\\
		\frac{C|x'|}{\delta^{2}(x^\prime)},& \alpha=3,\\
		\frac{C|x'|}{\delta(x^\prime)},& \alpha=4.
		\end{cases}
		\end{align*}
	\end{lemma}
	
	\begin{proof}
		We only consider the case of $\alpha=1$ for instance, since other cases are similar. 
		Recalling the definition of  ${\bf v}_{1}$ and $\bar{p}_1$ in \eqref{v1alpha} and \eqref{defp113D}, by direct calculations, we have, for $x\in\Omega_{2R}$,
		\begin{align}
		\partial_{x_1}{\bf v}_{1}^{(1)}&=\partial_{x_1}k(x)-\frac{24}{5}\frac{(\kappa_1+\kappa)x_1\big(\delta(x^\prime)-(\kappa_1-\kappa)x_1^2\big)}{\delta^2(x^\prime)}\Big(k^2(x)-\frac{1}{4}\Big)\nonumber\\
		&\quad+2k(x)\partial_{x_1}k(x)F_1(x),\label{estv112} \\
		\partial_{x_2}{\bf v}_{1}^{(1)}&=\partial_{x_2}k(x)+\frac{24(\kappa_1+\kappa)(\kappa_1-\kappa)}{5}\frac{x_1^2x_2}{\delta^2(x^\prime)}\Big(k^2(x)-\frac{1}{4}\Big)+2k(x)\partial_{x_2}k(x)F_1(x), \label{estv1121}\\
		\partial_{x_3}{\bf v}_{1}^{(1)}&=\frac{1}{\delta(x^\prime)}+\frac{2k(x)}{\delta(x_{1})}F_1(x),\label{estv1112}\\
		\partial_{x_i}{\bf v}_{1}^{(2)}&=-\frac{12}{5}\frac{(\kappa_1+\kappa)(x_j\delta(x^\prime)-2(\kappa_1-\kappa)x_1x_2x_i)}{5\delta^2(x^\prime)}\Big(k^2(x)-\frac{1}{4}\Big)\nonumber\\
		&\quad+2k(x)\partial_{x_i}k(x)G(x),\quad i,j=1,2, ~i\neq j,\nonumber\\
		\partial_{x_3}{\bf v}_{1}^{(2)}&=\frac{2k(x)}{\delta(x^\prime)}G(x),\quad|\partial_{x_j}{\bf v}_{1}^{(3)}|\leq C,\quad j=1,2,\nonumber\\
		\partial_{x_3}{\bf v}_{1}^{(3)}&=\frac{2}{\delta(x^\prime)}\bigg( (\kappa_1+\kappa)x_{1}+(\kappa_1-\kappa)\big(x_1F_1(x)+x_2G(x)\big)\bigg)\Big(k^2(x)-\frac{1}{4}\Big)\nonumber\\
		&\quad+\frac{2k(x)}{\delta(x^\prime)}{\bf E}_1^{(3)}(x),\label{estv11223}
		\end{align}
		where $F_1(x),G(x)$, and ${\bf E}_1(x)$ are defined in \eqref{defFGH}. These estimates imply that	
		\begin{align}\label{v1u1}
		\nabla\cdot{\bf v}_1=0,\quad |\nabla{\bf v}_{1}|\leq \frac{C}{\delta(x^\prime)}\quad\mbox{in}~\Omega_{2R}.
		\end{align}
		Furthermore, for the second-order partial derivatives, we have, for $i=1,2$,
		\begin{align}
		|\partial_{x_ix_i}{\bf v}_{1}^{(1)}|&\leq\frac{C}{\delta(x^\prime)},\quad\partial_{x_3x_3}{\bf v}_{1}^{(1)}=\frac{2}{\delta^2(x^\prime)}F_1(x),\label{v11-222-2D}\\
		|\partial_{x_ix_i}{\bf v}_{1}^{(2)}|&\leq\frac{C|x^\prime|}{\delta(x^\prime)},\quad \partial_{x_3x_3}{\bf v}_{1}^{(2)}=\frac{2}{\delta^2(x^\prime)}G(x),\label{v12-3D}\\
		|\partial_{x_ix_i}{\bf v}_{1}^{(3)}|&\leq\frac{C|x^\prime|}{\delta(x^\prime)},\quad|\partial_{x_3x_3}{\bf v}_{1}^{(3)}|\leq\frac{C|x^\prime|}{\delta^2(x^\prime)},\quad|\partial_{x_ix_3}{\bf v}_{1}^{(3)}|\leq\frac{C}{\delta(x^\prime)},\label{v12-222-2D}\\
		\partial_{x_1}\bar{p}_1&=\frac{2\mu}{\delta^2(x_{1})}F_1(x)+\mu\partial_{x_1x_3} {\bf v}_{1}^{(3)},\quad
		\partial_{x_2}\bar{p}_1=\frac{2\mu}{\delta^(x^\prime)}G(x)+\mu\partial_{x_2x_3} {\bf v}_{1}^{(3)}.\label{v11-22o2-2g}
		\end{align}
		It is easy to see from the definition of $\bar{p}_1$ in \eqref{defp113D} that
		$$\mu\partial_{x_3x_3}{\bf v}_{1}^{(3)}-\partial_{x_3}\bar{p}_1=0.$$
		Moreover, it follows from \eqref{v11-222-2D}--\eqref{v11-22o2-2g} that
		$$\mu\partial_{x_3x_3}{\bf v}_{1}^{(1)}-\partial_{x_1}\bar{p}_1=-\mu\partial_{x_1x_3}{\bf v}_{1}^{(3)},\quad \mu\partial_{x_3x_3}{\bf v}_{1}^{(2)}-\partial_{x_2}\bar{p}_1=-\mu\partial_{x_2x_3}{\bf v}_{1}^{(3)},$$
		and
		\begin{align}\label{estdivv11p1}
		\left|\mu\Delta{\bf v}_{1}-\nabla\bar{p}_1\right|\leq \frac{C}{\delta(x^\prime)}.
		\end{align}
		Lemma \ref{lemauxi} is proved.
	\end{proof}
	
	For $|x^\prime|\le R$, we define a constant independently of $\varepsilon$,
	\begin{equation}\label{defqalpha}
	(q_{\alpha})_{\Omega_R}=\frac{1}{|\Omega_{R}|}\int_{\Omega_{R}}q_{\alpha}(y)\mathrm{d}y,\quad\alpha=1,\dots,6,
	\end{equation}
	and denote
	\begin{equation}\label{Omegadel}
	\Omega_{\delta}(x^\prime):=\left\{(y^\prime,y_{3})\in\mathbb R^3\big| h(y^\prime)<y_{3}
	<\varepsilon+h_{1}(y^\prime),\,|y^\prime-x^\prime|<\delta \right\},
	\end{equation} 
	where $\delta=\delta(x')$. 
	
	\begin{prop}\label{propu11}
		Let ${\bf u}_{\alpha}\in{C}^{2}(\Omega;\mathbb R^3)$, $p_{\alpha}\in{C}^{1}(\Omega)$ be the solution to \eqref{equ_v1} with $\alpha=1,\dots,6$. Then in $\Omega_{R}$, 
		\begin{equation*}
		\|\nabla({\bf u}_{\alpha}-{\bf v}_{\alpha})\|_{L^{\infty}(\Omega_{\delta/2}(x'))}\leq 
		\begin{cases}
		C,&\alpha=1,2,5,6,\\
		\frac{C}{\sqrt{\delta(x^\prime)}},&\alpha=3,\\
		C\sqrt{\delta(x^\prime)},&\alpha=4,
		\end{cases}
		\end{equation*}
		and 
		\begin{equation*}
		\|\nabla^2({\bf u}_{\alpha}-{\bf v}_{\alpha})\|_{L^{\infty}(\Omega_{\delta/2}(x'))}+\|\nabla q_\alpha\|_{L^{\infty}(\Omega_{\delta/2}(x'))}\leq
		\begin{cases}
		\frac{C}{\delta(x^\prime)},&\alpha=1,2,5,6,\\
		\frac{C}{\delta^{3/2}(x^\prime)},&\alpha=3,\\
		\frac{C}{\sqrt{\delta(x^\prime)}},&\alpha=4.
		\end{cases}
		\end{equation*}
		Consequently, in $\Omega_{R}$, we have 
		\begin{align*}
		|\nabla {\bf u}_{\alpha}(x)|\leq 
		\begin{cases}
		\frac{C}{\delta(x^\prime)},&\alpha=1,2,5,6,\\
		C\left(\frac{1}{\delta(x^\prime)} +\frac{|x^\prime|}{\delta^2(x^\prime)}\right),&\alpha=3,\\
		C\left(\frac{|x^\prime|}{\delta(x^\prime)} +1\right),&\alpha=4,
		\end{cases}
		\end{align*}
		$$|p_{\alpha}(x)-(q_{\alpha})_{\Omega_R}|\leq
		\begin{cases}
		\frac{C}{\varepsilon^{3/2}},&\alpha=1,2,5,6,\\
		\frac{C}{\varepsilon^{2}},&\alpha=3,\\
		\frac{C}{\sqrt{\varepsilon}},&\alpha=4,
		\end{cases}$$
		and
		\begin{align*}
		|\nabla^2 {\bf u}_{\alpha}(x)|+|\nabla p_{\alpha}(x)|\leq 
		\begin{cases}
		\frac{C}{\delta^2(x^\prime)},&\alpha=1,2,5,6,\\
		C\left(\frac{1}{\delta^2(x^\prime)} +\frac{|x^\prime|}{\delta^3(x^\prime)}\right),&\alpha=3,\\
		\frac{C}{\delta(x^\prime)},&\alpha=4.
		\end{cases}
		\end{align*}
	\end{prop}
	
	\begin{proof}
		We only consider the case $\alpha=1$ since the proof of the case $\alpha=2,\dots,6$ is similar. For simplicity, we denote ${\bf v}:={\bf v}_1$, $\bar p:=\bar p_1$, ${\bf u}:={\bf u}_1$, $p:=p_1$, and set ${\bf w}={\bf u}-{\bf v}$, $q=p-\bar p$. Then it follows from \eqref{equ_v1} that $({\bf w},q)$ satisfies
		\begin{align}\label{w1}
		\begin{cases}
		-\mu\,\Delta{\bf w}+\nabla q={\bf f}:=\mu\,\Delta{\bf v}-\nabla\bar{p}\quad&\mathrm{in}\,\Omega,\\
		\nabla\cdot {\bf w}=0\quad&\mathrm{in}\,\Omega_{2R},\\
		\nabla\cdot {\bf w}=-\nabla\cdot {\bf v}\quad&\mathrm{in}\,\Omega\setminus\Omega_R,\\
		{\bf w}=0\quad&\mathrm{on}\,\partial\Omega,
		\end{cases}
		\end{align}
		where $\Omega_R$ is defined in \eqref{narrow-region}. We shall divide the proof into three steps.
		
		{\bf Step 1. Global boundedness of $|\nabla{\bf w}|$: 
			\begin{align}\label{w1alpha}
			\int_{\Omega}|\nabla{\bf w}|^{2}\mathrm{d}x\leq\,C.
			\end{align}}  
		
		We first recall that if 
		\begin{align}\label{int-fw}
		\Big| \int_{\Omega_{R}}\sum_{j=1}^{3}{\bf f}^{(j)}{\bf w}^{(j)}\mathrm{d}x\Big|\leq\,C\left(\int_{\Omega}|\nabla {\bf w}|^2\mathrm{d}x\right)^{1/2},
		\end{align}
		then \eqref{w1alpha} follows. See \cite[Lemma 3.7]{LX} for details. To prove \eqref{int-fw}, using the Sobolev trace embedding theorem, we have 
		\begin{align}\label{w1Dw1}
		\int_{\substack{|x'|=R,\\h(x')<x_{3}<\varepsilon+h_{1}(x')}}|{\bf w}|\mathrm{d}x_{3}
		&\leq C\left(\int_{\Omega}|\nabla {\bf w}|^2\mathrm{d}x\right)^{1/2}.
		\end{align}
		By using \eqref{estdivv11p1},  the  integration by parts with respect to $x'$, and  \eqref{w1Dw1}, we obtain
		\begin{align*}
		\left|\int_{\Omega_{R}} {\bf f}^{(1)}{\bf w}^{(1)}\mathrm{d}x\right|
		&=\left|\int_{\Omega_{R}} \mu{\bf w}^{(1)}\big(\partial_{x_1x_1}{\bf v}^{(1)}+\partial_{x_2x_2}{\bf v}^{(1)}-\partial_{x_1x_3}{\bf v}^{(3)}\big)\mathrm{d}x\right|\nonumber\\
		&\leq\,C\int_{\Omega_{R}}|\nabla_{x'}{\bf w}^{(1)}|\big(|\partial_{x_1}{\bf v}^{(1)}|+|\partial_{x_2}{\bf v}^{(1)}|+|\partial_{x_3}{\bf v}^{(3)}|\big)\mathrm{d}x\nonumber\\
		&\quad+C\,\int_{\substack{|x'|=R,\\h(x')<x_{3}<\varepsilon+h_{1}(x')}}|{\bf w}^{(1)}|\mathrm{d}x_{3}\\
		&\leq C\int_{\Omega_{R}}|\nabla_{x'}{\bf w}^{(1)}|\big(|\partial_{x_1}{\bf v}^{(1)}|+|\partial_{x_2}{\bf v}^{(1)}|+|\partial_{x_3}{\bf v}^{(3)}|\big)\mathrm{d}x\nonumber\\
		&\quad+C\left(\int_{\Omega}|\nabla {\bf w}|^2\mathrm{d}x\right)^{1/2}=:\mathrm{I}_{11}+\mathrm{I}_{12}.
		\end{align*}
		From  \eqref{estv112}, \eqref{estv1121},  and \eqref{estv11223}, it follows that
		\begin{equation*}
		|\partial_{x_1}{\bf v}^{(1)}|,|\partial_{x_2}{\bf v}^{(1)}|,|\partial_{x_3}{\bf v}^{(3)}|\leq\frac{C|x'|}{\delta(x')}\quad
		\mbox{ in} ~\Omega_{2R}.
		\end{equation*}
		Combining with H\"{o}lder's inequality, we have
		\begin{align*}
		|\mathrm{I}_{11}|&\leq C\left(\int_{\Omega_{R}}\big(|\partial_{x_1}{\bf v}^{(1)}|^{2}+|\partial_{x_2}{\bf v}^{(1)}|^{2}+|\partial_{x_3}{\bf v}^{(3)}|^2\big)\mathrm{d}x\right)^{1/2}
		\left(\int_{\Omega} |\nabla_{x'}{\bf w}^{(1)}|^2\mathrm{d}x\right)^{1/2}\nonumber\\
		&\leq C \left(\int_{\Omega} |\nabla {\bf w}|^2\mathrm{d}x\right)^{1/2}.
		\end{align*}
		Thus, we conclude that
		\begin{align*}
		\left|\int_{\Omega_{R}} {\bf f}^{(1)}{\bf w}^{(1)}\mathrm{d}x\right|
		\leq C\left(\int_{\Omega}|\nabla {\bf w}|^2\mathrm{d}x\right)^{1/2}.
		\end{align*}
		
		Similarly, using \eqref{v12-3D}, \eqref{v12-222-2D}, \eqref{estdivv11p1},  and \eqref{w1Dw1}, we obtain
		\begin{align*}
		\left|\int_{\Omega_{R}} {\bf f}^{(2)}{\bf w}^{(2)}\mathrm{d}x\right|\leq  C \left(\int_{\Omega} |\nabla {\bf w}|^2\mathrm{d}x\right)^{1/2},
		\end{align*}
		and
		\begin{align*}
		&\left|\int_{\Omega_{R}} {\bf f}^{(3)}{\bf w}^{(3)}\mathrm{d}x\right|=\left|\mu\,\int_{\Omega_{R}} {\bf w}^{(3)}\big(\partial_{x_1x_1}{\bf v}^{(3)}+\partial_{x_2x_2}{\bf v}^{(3)}\big)\mathrm{d}x\right|\nonumber\\
		\leq&\,C\int_{\Omega_{R}}|\nabla_{x'}{\bf w}^{(3)}|\big(|\partial_{x_1}{\bf v}^{(3)}|+|\partial_{x_2}{\bf v}^{(3)}|\big)\mathrm{d}x+C\int_{\substack{|x'|=R,\\h(x_1)<x_{3}<\varepsilon+h_{1}(x')}}|{\bf w}^{(3)}|\mathrm{d}x_{3}\\
		\leq&\,  C \left(\int_{\Omega} |\nabla {\bf w}|^2\mathrm{d}x\right)^{1/2}.
		\end{align*}
		Therefore, we have \eqref{int-fw}.
		
		{\bf Step 2. Local estimate of $|\nabla{\bf w}|$:}
		\begin{align}\label{estw11narrow}
		\int_{\Omega_{\delta}(z')}|\nabla {\bf w}|^{2}\mathrm{d}x\leq C\delta^3(z'),
		\end{align}
		where $\delta:=\delta(z')$, $|z'|\leq R$, and $\Omega_{\delta}(z')$ is defined by \eqref{Omegadel}. Recall that
		\begin{align}\label{iterating1}
		\int_{\Omega_{t}(z')}|\nabla {\bf w}|^{2}\mathrm{d}x\leq\,&
		\frac{C\delta^2(z')}{(s-t)^{2}}\int_{\Omega_{s}(z')}|\nabla {\bf w}|^2\mathrm{d}x+C\left((s-t)^{2}+\delta^{2}(z')\right)\int_{\Omega_{s}(z')}|{\bf f}|^{2}\mathrm{d}x.
		\end{align}
		See \cite[Lemma 3.10]{LX}. Denote
		$$F(t):=\int_{\Omega_{t}(z')}|\nabla {\bf w}|^{2}\mathrm{d}x.$$
		From 	\eqref{estdivv11p1}, it follows that
		\begin{align*}
		\int_{\Omega_{s}(z')}|{\bf f}|^{2}\mathrm{d}x\leq\frac{Cs^2}{\delta(z')}.
		\end{align*}
		Substituting it into \eqref{iterating1}, we have 
		\begin{align}\label{iteration3D}
		F(t)\leq \left(\frac{c_{0}\delta(z')}{s-t}\right)^2 F(s)+C\left((s-t)^{2}+\delta^2(z'\right)\frac{s^2}{\delta(z')},
		\end{align}
		where $c_{0}$ is a constant to be fixed.  Let $k_{0}=\left[\frac{1}{4c_{0}\sqrt{\delta(z')}}\right]+1$ and $t_{i}=\delta(z')+2c_{0}i\delta(z'), i=0,1,2,\dots,k_{0}$. Then we take $s=t_{i+1}$ and $t=t_{i}$ in \eqref{iteration3D} to get 
		\begin{align*}
		F(t_{i})\leq \frac{1}{4}F(t_{i+1})+C(i+1)\delta^3(z').
		\end{align*}
		After $k_{0}$ iterations, and using \eqref{w1alpha}, we obtain, for sufficiently small $\varepsilon$ and $|z'|$,
		\begin{align*}
		F(t_0)\leq \left(\frac{1}{4}\right)^{k}F(t_{k})
		+C\delta^3(z')\sum\limits_{l=0}^{k-1}\left(\frac{1}{4}\right)^{l}(l+1)\leq C\delta^3(z').
		\end{align*}
		This gives \eqref{estw11narrow}.
		
		{\bf Step 3. Local $W^{1,\infty}$ and $W^{2,\infty}$  estimates.} With \eqref{estdivv11p1} and \eqref{estw11narrow}, we are ready to get the desired estimates. 
		We  obtain from \cite[Proposition 3.6]{LX}, \eqref{estdivv11p1}, and \eqref{estw11narrow} that 
		\begin{align*}
		\|\nabla {\bf w}\|_{L^{\infty}(\Omega_{\delta/2}(z'))}
		\leq
		\frac{C}{\delta^{3/2}(z')}\|\nabla{\bf w}\|_{L^{2}(\Omega_{\delta}(z'))}+C\delta(z')\|{\bf f}\|_{L^{\infty}(\Omega_{\delta}(z'))}\leq C,
		\end{align*}
		and
		\begin{align}\label{D2wDq}
		&\|\nabla^2 {\bf w}\|_{L^{\infty}(\Omega_{\delta/2}(z'))}+\|\nabla q\|_{L^{\infty}(\Omega_{\delta/2}(z'))}\nonumber\\
		&\leq
		\frac{C}{\delta^{5/2}(z')}\|\nabla{\bf w}\|_{L^{2}(\Omega_{\delta}(z'))}+C\|{\bf f}\|_{L^{\infty}(\Omega_{\delta}(z'))}+C\delta(z')\|\nabla{\bf f}\|_{L^{\infty}(\Omega_{\delta}(z'))}\leq \frac{C}{\delta(z')}.
		\end{align}
		Recalling ${\bf w}={\bf u}-{\bf v}$, it follows from \eqref{v1u1} that
		\begin{align*}
		|\nabla {\bf u}|\leq|\nabla {\bf v}(z)|+|\nabla {\bf w}(z)|\leq \frac{C}{\delta(z')},\quad z\in\Omega_{R},
		\end{align*}
		and using \eqref{v11-222-2D}--\eqref{v12-222-2D}, we have 
		\begin{align*}
		|\nabla^2 {\bf u}|\leq \frac{C}{\delta(z')}+\frac{C|z'|}{\delta^2(z')},\quad z\in\Omega_{R}.
		\end{align*}
		By $q=p-\bar{p}$, \eqref{defp113D},  the mean value theorem, and the estimate of $|\nabla q|$ in \eqref{D2wDq}, we have
		\begin{align*}
		|p(z)-(q)_{\Omega_R}|\leq |q-(q)_{\Omega_R}|+|\bar{p}|\leq \frac{C}{\varepsilon}+\frac{C|z'|}{\delta^2(z')}\leq\frac{C}{\varepsilon^{3/2}},\quad z\in\Omega_{R}.
		\end{align*}
		Thus, we finish the proof of Proposition \ref{propu11}.
	\end{proof}
	
	\section{Estimates of $(\nabla{\bf u}_0,p_0)$ in 3D}\label{subsec22}
	
	In this section we estimate $|\nabla{\bf u}_{0}|$ and $p_0$ satisfying \eqref{equ_v3}, according to the given different kinds of boundary data $\boldsymbol\varphi$. Define 
	$$(q_0^{\alpha})_{\Omega_R}=\frac{1}{|\Omega_{R}|}\int_{\Omega_{R}}q_0^{\alpha}(x)\mathrm{d}x,\quad\alpha=1,2,3,$$
	which is a constant independent of $\varepsilon$.
	
	(a) For $\boldsymbol\varphi\in{\bf\Phi}_1$. Here we only present the details for the case ${\boldsymbol\varphi}=(x_1^{l_1},0,0)^{\mathrm T}$ on $\Gamma_{2R}$, the other cases are similar. 
	
	(a1) If $l_1=0$, for the locally constant boundary value $\boldsymbol\varphi=(1,0,0)^{\mathrm T}$, we choose ${\bf v}_{0}^1\in C^{2}(\Omega;\mathbb R^3)$ satisfying
	\begin{equation}\label{divfree1}
	{\bf v}_0^{1}+{\bf v}_1=(1,0,0)^{\mathrm T},\quad\mbox{in}~\Omega_{2R},
	\end{equation}
	where ${\bf v}_1$ is defined by \eqref{v1alpha}. 
	We construct $\bar{p}_0^1\in C^{1}(\Omega)$ such that
	\begin{equation*}
	\bar{p}_0^1=-\bar{p}_1=-\frac{6\mu}{5}\frac{(\kappa_1+\kappa)}{(\kappa_1-\kappa)}\frac{ x_{1}}{\delta^2(x^\prime)}+\mu\partial_{x_3} ({\bf v}_0^1) ^{(3)},\quad\mbox{in}~\Omega_{2R}.
	\end{equation*}
	Clearly, $\nabla\cdot{\bf v}_0^1=0$ in $\Omega_{2R}$, and the properties of ${\bf v}_1$ is also useful here.
	
	(a2) If $l_1\geq1$, for the polynomial cases, we seek ${\bf v}_{0}^1\in C^{2}(\Omega;\mathbb R^3)$ satisfying, in $\Omega_{2R}$,
	\begin{align*}
	{\bf v}_0^1=\boldsymbol\varphi\big(\frac{1}{2}-k(x)\big)+\big({\bf E}_0^{1(1)},0,{\bf E}_0^{1(3)}\big)^{\mathrm T}
	\Big(k^2(x)-\frac{1}{4}\Big),
	\end{align*}
	where
	\begin{align*}
	{\bf E}_0^{1(1)}&=\frac{32(\kappa_1-\kappa)}{l_1+2}\frac{x_1^{l_1+2}}{\delta(x')}k(x)+\frac{12\kappa }{l_1+2}\frac{x_1^{l_1+2}}{\delta(x^\prime)}-8x_1^{l_1}k(x)+3x_1^{l_1},\\      
	{\bf E}_0^{1(3)}&=(2k(x)-1)k(x)\delta(x')l_1x_1^{l_1-1}-2x_1^{l_1+1}k(x)\Big((\kappa_1+\kappa)+2(\kappa_1-\kappa)k(x)\Big)\\&\quad-\delta(x^\prime)\partial_{x_1}k(x){\bf E}_0^{1(1)}.
	\end{align*} 
	Then $\nabla\cdot{\bf v}_0^1=0$ in $\Omega_{2R}$, and  the biggest term is $\partial_{x_3x_3}({\bf v}_0^1)^{(1)}\sim \frac{|x^\prime|^{l_1}}{\delta^2(x\prime)}$. It is a good term. So we can directly choose  $\bar{p}_0^1=0$  in $\Omega$. A direct calculation gives, 
	\begin{align}\label{estf01}
	\left|\mu\Delta{\bf v}_0^1-\nabla\bar{p}_0^1\right|=\left|\mu\Delta{\bf v}_0^1\right|\leq         
	\begin{cases}
	\frac{C|x^\prime|}{\delta^2(x^\prime)},&\mbox{if}~l_1=1,\\
	\frac{C}{\delta(x^\prime)},&\mbox{if}~l_1\geq2,
	\end{cases}
	\quad \mbox{in}~\Omega_{2R}.
	\end{align} 
	
	Denote
	\begin{align*}
	\rho_1^{l_1}=
	\begin{cases}
	\frac{1}{\sqrt{\delta(x')}},&\quad l_1=1;\\
	1,&\quad l_1\neq 1. 
	\end{cases}
	\end{align*} 
	Then applying the energy method  and  iteration technique as in the proof of Proposition \ref{propu11}, we derive the following result.         
	
	\begin{prop}\label{propu15}
		Let ${\bf u}_0\in{C}^{2}(\Omega;\mathbb R^3),~p_0\in{C}^{1}(\Omega)$ be the solution to \eqref{equ_v3}. Then in $\Omega_R$, we have 
		\begin{equation*}
		\|\nabla({\bf u}_0-{\bf v}_0^1)\|_{L^{\infty}(\Omega_{\delta/2}(x'))}\leq C\rho_1^{l_1},
		\end{equation*}
		and 
		\begin{equation*}
		\|\nabla^2({\bf u}_0-{\bf v}_0^1)\|_{L^{\infty}(\Omega_{\delta/2}(x'))}+\|\nabla q_0^1\|_{L^{\infty}(\Omega_{\delta/2}(x'))}\leq C
		\begin{cases}
		\frac{C}{\delta^{3/2}(x')},&~ l_1=1,\\
		\frac{1}{\delta(x')},&~l_1\neq 1.
		\end{cases}
		\end{equation*}
		Consequently, for $x\in\Omega_{R}$,
		\begin{align*}
		|\nabla {\bf u}_0(x)|\leq C\rho_1^{l_1}+\frac{C|x'|^{l_1}}{\delta(x')},\quad|p_0(x)-(q_0^1)_{\Omega_R}|\leq
		\begin{cases}
		\frac{C}{\varepsilon^{3/2}},&l_1=0,1,\\
		\frac{1}{\varepsilon},&~l_1\geq2,
		\end{cases}
		\end{align*}	
		and 
		\begin{align*}
		|\nabla^2{\bf u}_0(x)|+|\nabla p_0(x)|\leq
		\begin{cases}
		\frac{C}{\delta^2(x^\prime)},&l_1=0,\\
		\frac{C}{\delta^{3/2}(x')},&l_1=1,\\
		\frac{1}{\delta(x')},&~l_1\geq2.
		\end{cases}
		\end{align*}	
	\end{prop}         
	
	\begin{proof}
		The proof is an adaptation of the proof of Proposition \ref{propu11}. One can see from \eqref{divfree1} that the case of $l_1=0$ is essentially the same as that of Proposition \ref{propu11} with $\alpha=1$. Moreover, the proof for $l_1\geq2$ is also similar to the case of $\alpha=1$ in Proposition \ref{propu11}, since the estimate \eqref{estf01} when $l_1\geq2$ is the same as that in \eqref{estdivv11p1}. Thus,  we will only work on the case of $l_1=1$ in the following. As before, we denote ${\bf v}:={\bf v}_0^1$, $\bar p:=\bar p_0$, ${\bf u}:={\bf u}_0$, $p:=p_0$, and set ${\bf w}={\bf u}-{\bf v}$, $q=p-\bar p$. Then it follows from \eqref{equ_v3} that $({\bf w},q)$ satisfies \eqref{w1}.
		
		A direct calculation yields
		\begin{equation*}
		|\nabla{\bf v}|\leq \frac{C|x'|^{l_1}}{\delta(x')}.
		\end{equation*}
		Then using the argument in the proof of Proposition \ref{propu11} (see Step 1 for the details), we find that \eqref{int-fw} holds true. By \eqref{estf01}, the estimate \eqref{iteration3D} becomes
		\begin{align*}
		F(t)\leq \left(\frac{c_{1}\delta(z')}{s-t}\right)^2 F(s)+C\left((s-t)^{2}+\delta^2(z'\right)\frac{s^2}{\delta^3(z')}(s+|z'|^2),
		\end{align*}
		where $c_{1}$ is a constant to be fixed. Then applying the iteration technique as used in the proof of \eqref{estw11narrow}, we obtain 
		\begin{align*}
		\int_{\Omega_{\delta}(z')}|\nabla {\bf w}|^{2}\mathrm{d}x\leq C\delta^2(z'),
		\end{align*}
		where $\delta:=\delta(z')$, $|z'|\leq R$, and $\Omega_{\delta}(z')$ is defined in \eqref{Omegadel}. This together with \cite[Proposition 3.6]{LX} and \eqref{estf01} gives
		\begin{align*}
		\|\nabla {\bf w}\|_{L^{\infty}(\Omega_{\delta/2}(z'))}
		\leq
		\frac{C}{\delta^{3/2}(z')}\|\nabla{\bf w}\|_{L^{2}(\Omega_{\delta}(z'))}+C\delta(z')\|{\bf f}\|_{L^{\infty}(\Omega_{\delta}(z'))}\leq \frac{C}{\sqrt{\delta(z')}},
		\end{align*}
		and
		\begin{align*}
		&\|\nabla^2 {\bf w}\|_{L^{\infty}(\Omega_{\delta/2}(z'))}+\|\nabla q\|_{L^{\infty}(\Omega_{\delta/2}(z'))}\nonumber\\
		&\leq
		\frac{C}{\delta^{5/2}(z')}\|\nabla{\bf w}\|_{L^{2}(\Omega_{\delta}(z'))}+C\|{\bf f}\|_{L^{\infty}(\Omega_{\delta}(z'))}+C\delta(z')\|\nabla{\bf f}\|_{L^{\infty}(\Omega_{\delta}(z'))}\leq \frac{C}{\delta^{3/2}(z')}.
		\end{align*}
		The proof is finished.
	\end{proof}
	
	(b) For $\boldsymbol\varphi\in{\bf\Phi}_2$. Here we only consider the case $i=1$ in \eqref{defphi2}, since the case $i=2$ is the same. 
	
	(b1) If $l_2=0$, for the locally constant $\boldsymbol\varphi=(0,0,1)^{\mathrm T}$,  then we choose ${\bf v}_{0}^2\in C^{2}(\Omega;\mathbb R^3)$ satisfying
	\begin{equation*}
	{\bf v}_0^2+{\bf v}_3=(0,0,1)^{\mathrm T} \quad \mbox{in}~\Omega_{2R}.
	\end{equation*}
	We construct $\bar{p}_0^2\in C^{1}(\Omega)$ such that, in $\Omega_{2R}$,
	\begin{equation*}
	\bar{p}_0^2=-\frac{3\mu}{2(\kappa_1-\kappa)\delta^2(x^\prime)}+\mu\partial_{x_3} ({\bf v}_0^2) ^{(3)} .
	\end{equation*}         
	
	(b2) If $l_2=1$, we construct ${\bf v}_{0}^2\in C^{2}(\Omega;\mathbb R^3)$ satisfying, in $\Omega_{2R}$,
	\begin{align*}
	{\bf v}_0^2=\boldsymbol\varphi\big(\frac{1}{2}-k(x)\big)+\big({\bf E}_0^{2(1)},{\bf E}_0^{2(2)},{\bf E}_0^{2(3)}\big)^{\mathrm T}
	\Big(k^2(x)-\frac{1}{4}\Big)\quad \mbox{in}~\Omega_{2R},
	\end{align*}           
	where                   
	\begin{align*}
	{\bf E}_0^{2(1)}=&-\frac{12}{5}\frac{x_1^2}{\delta(x^\prime)}+\frac{3}{5(\kappa_1-\kappa)},\quad{\bf E}_0^{2(2)}=-\frac{12}{5}\frac{x_1x_2}{\delta(x^\prime)},\\
	{\bf E}_0^{2(3)}=&2k(x)x_1-\delta(x^\prime)\partial_{x_1}k(x){\bf E}_0^{2(1)}-\delta(x^\prime)\partial_{x_2}k(x){\bf E}_0^{2(2)}.
	\end{align*}
	
	(b3) If $l_2\geq 2$, we seek ${\bf v}_{0}^2\in C^{2}(\Omega;\mathbb R^3)$ satisfying, in $\Omega_{2R}$,
	\begin{align*}
	{\bf v}_0^2=\boldsymbol\varphi\big(\frac{1}{2}-k(x)\big)+\big({\bf E}_0^{2(1)},0,{\bf E}_0^{2(3)}\big)^{\mathrm T}
	\Big(k^2(x)-\frac{1}{4}\Big)\quad \mbox{in}~\Omega_{2R}.
	\end{align*}
	Here
	$$  {\bf E}_0^{2(1)}= -\frac{6}{l_2+1}\frac{x_1^{l_2+1}}{\delta(x^\prime)},\quad{\bf E}_0^{2(3)}=2k(x)x_1-\delta(x^\prime)\partial_{x_1}k(x){\bf E}_0^{2(1)}.$$
	It also holds that $\nabla\cdot{\bf v}_0^2=0$ in $\Omega_{2R}$ and 
	\begin{align*}
	|\nabla{\bf v}_0^2|\leq 
	\begin{cases}
	\frac{C}{\delta(x^\prime)}+\frac{|x^\prime|}{\delta^2(x^\prime)},&\quad l_2=0,\\
	\frac{C|x^\prime|^{l_2-1}}{\delta(x^\prime)},&\quad l_2\geq1.
	\end{cases}
	\end{align*}
	In order to control the biggest term in $\Delta{\bf v}_0^2$, we take
	\begin{align*}
	\bar{p}_0^2=
	\begin{cases}
	\frac{6\mu }{5(\kappa_1-\kappa)}\frac{ x_1 }{\delta^2(x')}+\mu \partial_{x_3}({\bf v}_0^2)^{(3)},&\quad l_2=1;\\
	0,&\quad l_2\geq 2.
	\end{cases}
	\end{align*}
	Furthermore, we deduce, in $\Omega_{2R}$,
	\begin{align*}
	\left|\mu\Delta{\bf v}_0^2-\nabla\bar{p}_0^2\right|\leq
	\begin{cases}
	\frac{C}{\delta(x^\prime)},&\quad l_2=1,~l_2\geq3,\\
	\frac{C|x^\prime|}{\delta^2(x^\prime)},&\quad l_2=0,2.
	\end{cases}
	\end{align*} 
	
	Let
	\begin{align*}
	\rho_2^{l_2}=
	\begin{cases}
	\frac{1}{\sqrt{\delta(x')}},&\quad l_2=0, 2;\\
	1,&\quad l_2=1,~l_2\geq 3. 
	\end{cases}
	\end{align*}
	Then apply the same argument as used in the proof of Proposition \ref{propu15}, we derive Proposition \ref{propu16} as follows.
	
	\begin{prop}\label{propu16}
		Let ${\bf u}_0\in{C}^{2}(\Omega;\mathbb R^3),~p_0\in{C}^{1}(\Omega)$ be the solution to \eqref{equ_v3}. Then in $\Omega_{R}$,
		\begin{align*}
		\|\nabla({\bf u}_0-{\bf v}_0^2)\|_{L^{\infty}(\Omega_{\delta/2}(x'))}\leq C\rho_{2}^{l_2},
		\end{align*}
		and 
		\begin{align*}
		\|\nabla^2({\bf u}_0-{\bf v}_0^2)\|_{L^{\infty}(\Omega_{\delta/2}(x'))}+\|\nabla q_0^2\|_{L^{\infty}(\Omega_{\delta/2}(x'))}\leq 
		\frac{C}{\delta(x^\prime)}+\frac{C|x^\prime|}{\delta^2(x^\prime)}.
		\end{align*}
		Consequently, for $x\in\Omega_R$,
		\begin{align*}
		|\nabla {\bf u}_0(x)|\leq \begin{cases}
		\frac{C}{\delta(x^\prime)}+\frac{C|x^\prime|}{\delta^2(x^\prime)},&\quad l_2=0,\\
		\frac{C}{\delta(x^\prime)},&\quad l_2=1,\\
		\frac{C}{\sqrt{\delta(x^\prime)}}+\frac{C|x^\prime|}{\delta(x^\prime)},&\quad l_2=2,\\
		C,&\quad l_2\geq 3,
		\end{cases}
		\end{align*}
		\begin{align*}
		|p_0(x)-(q_0^2)_{\Omega_R}|\leq
		\begin{cases}
		\frac{C}{\varepsilon^2},&\quad l_2=0,\\
		\frac{C}{\varepsilon^{3/2}},&\quad l_2\geq 1,
		\end{cases}
		\end{align*}
		and 
		\begin{align*}
		|\nabla^2 {\bf u}_0(x)|+|\nabla p_0(x)|\leq
		\begin{cases}
		\frac{C}{\delta^2(x^\prime)}+\frac{C|x'|}{\delta^3(x^\prime)},&\quad l_2=0,\\
		\frac{C}{\delta^2(x^\prime)},&\quad l_2=1,\\
		\frac{C}{\delta(x^\prime)}+\frac{C|x^\prime|}{\delta^2(x^\prime)},&\quad l_2\geq2.
		\end{cases}
		\end{align*}
	\end{prop}
	
	(c) If $\boldsymbol\varphi\in{\bf\Phi}_3$ with $j=1$ in \eqref{defphi3},  we seek ${\bf v}_{0}^3\in C^{2}(\Omega;\mathbb R^3)$ satisfying, in $\Omega_{2R}$,
	\begin{align*}
	{\bf v}_0^3=\boldsymbol\varphi\big(\frac{1}{2}-k(x)\big)+\begin{pmatrix}
	2x_3^{l_3}k(x)+\frac{6}{l_3+1}\frac{x_3^{l_3+1}}{\delta(x^\prime)}\\\\
	0\\\\
	{\bf E}_0^{3(3)}
	\end{pmatrix}
	\Big(k^2(x)-\frac{1}{4}\Big)\quad \mbox{in}~\Omega_{2R}.
	\end{align*}
	Here
	$$ {\bf E}_0^{3(3)}=\big((\kappa_1+\kappa)x_1+2(\kappa_1-\kappa)x_1k(x)\big)\frac{6}{l_3+1}\frac{x_3^{l_3+1}}{\delta(x^\prime)}.$$
	Using a direct computation, we have $\nabla\cdot{\bf v}_0^3=0$ in $\Omega_{2R}$, $|\nabla{\bf v}_0^3|\leq C$, and $|\Delta {\bf v}_0^3|\leq \frac{C}{\delta(x^\prime)}$, so that we choose ${\bar p}_0^3=0$. 
	Also, in $\Omega_{2R}$,
	\begin{align*}
	\left|\mu\Delta{\bf v}_0^3-\nabla\bar{p}_0^3\right|\leq \frac{C}{\delta(x^\prime)}.
	\end{align*}
	
	\begin{prop}\label{propu162}
		Let ${\bf u}_{0}\in{C}^{2}(\Omega;\mathbb R^3),~p_{0}\in{C}^{1}(\Omega)$ be the solution to \eqref{equ_v3}. Then there holds
		\begin{equation*}
		\|\nabla({\bf u}_{0}-{\bf v}_{0}^3)\|_{L^{\infty}(\Omega_{\delta/2}(x'))}\leq C,
		\end{equation*}
		and 
		\begin{equation*}
		\|\nabla^2({\bf u}_{0}-{\bf v}_{0}^3)\|_{L^{\infty}(\Omega_{\delta/2}(x'))}+\|\nabla q_0^3\|_{L^{\infty}(\Omega_{\delta/2}(x'))}\leq \frac{C}{\delta(x^\prime)},~~ \,x\in\Omega_{R}.
		\end{equation*}
		Consequently, for $x\in\Omega_{R}$, 
		\begin{align*}
		|\nabla {\bf u}_{0}(x)|\leq C,\quad |p_{0}(x)-(q_{0}^3)_{\Omega_R}|\leq\frac{C}{\varepsilon},\quad|\nabla p_{0}(x)|+|\nabla^2{\bf u}_{0}(x)|\leq \frac{C}{\delta(x^\prime)}.
		\end{align*}
	\end{prop}           
	
	The auxiliary functions for $j=2,3$ in \eqref{defphi3} are constructed similarly, we left it to the interested reader.

	\section{Estimates of $C^{\alpha}$}\label{sec4}
	In this  section, we proceed to estimate $C^\alpha$, by using the estimates of $(\nabla{\bf u}_\alpha,p_\alpha)$, $\alpha=1,\dots,6$, and $(\nabla{\bf u}_0,p_0)$. We have the following Proposition.
	
	\begin{prop}\label{propu17}
		(1) Assume that ${\boldsymbol\varphi}\in{\bf\Phi}_1$, where ${\bf\Phi}_1$ is defined in \eqref{defphi1}. Then
		if $l_1=0$, we have
		\begin{align*}
		|C^1-1|, |C^\alpha|\leq\frac{C}{|\ln\varepsilon|},~\alpha=2,5,6,\quad |C^3|\leq C\varepsilon,\quad
		|C^4|\leq C;
		\end{align*}
		if $l_1\geq 1$, we have
		\begin{align*}
		|C^\alpha|\leq\frac{C}{|\ln\varepsilon|},~\alpha=1,2,5,6,\quad |C^3|\leq C\varepsilon,\quad	|C^4|\leq C.
		\end{align*}
		
		(2) Assume that ${\boldsymbol\varphi}\in{\bf\Phi}_2$, where ${\bf\Phi}_2$ is defined in \eqref{defphi2}. Then if $l_2=0$, we have
		\begin{align*}
		|C^\alpha|\leq\frac{C}{|\ln\varepsilon|},~\alpha=1,2,5,6,\quad |C^3-1|\leq C\varepsilon, \quad	|C^4|\leq C;
		\end{align*}
		if $l_2=1$, we have 
		\begin{align*}
		|C^\alpha|\leq C, ~\alpha=1,4,5,\quad  |C^\alpha|\leq\frac{C}{|\ln\varepsilon|},~\alpha=2,6,\quad |C^3|\leq C\varepsilon;
		\end{align*}
		if $l_2\geq 2$, we have
		\begin{align*}
		|C^\alpha|\leq\frac{C}{|\ln\varepsilon|},~\alpha=1,2,5,6,\quad |C^3|\leq C\varepsilon,\quad	|C^4|\leq C.
		\end{align*}
		
		(3) Assume that ${\boldsymbol\varphi}\in{\bf\Phi}_3$, where ${\bf\Phi}_3$ is defined in \eqref{defphi3}. Then we have
		\begin{align*}
		|C^\alpha|\leq\frac{C}{|\ln\varepsilon|},~\alpha=1,2,5,6,\quad |C^3|\leq C\varepsilon,\quad	|C^4|\leq C.
		\end{align*}
	\end{prop}
	
	From \eqref{ce}, to prove Proposition \ref{propu17}, we need to establish the following estimates and asymptotic expansions of $a_{\alpha\beta}$ and $Q_\beta[{\boldsymbol\varphi}]$. 
	
	\subsection{Estimates of $a_{\alpha\beta}$.}
	By \eqref{aijbj}, using the integration by parts, we have 
	\begin{align}\label{defaij}
	a_{\alpha\beta}=\int_{\Omega} \left(2\mu e({\bf u}_{\alpha}), e({\bf u}_{\beta})\right)\mathrm{d}x,\quad
	Q_{\beta}[{\boldsymbol\varphi}]=-\int_{\Omega} \left(2\mu e({\bf u}_{0}),e({\bf u}_{\beta})\right)\mathrm{d}x.
	\end{align}
	We will make use of the estimates of $|\nabla{\bf u}_\alpha|$ in Proposition \ref{propu11} to derive the following asymptotic expansions and estimates for  $a_{\alpha\beta}$. Set
	$$K(\kappa_1,\kappa)=\frac{9}{25}\Big(\frac{\kappa_1+\kappa}{\kappa_1-\kappa}\Big)^2+\frac{72}{25}\frac{\kappa_1\kappa}{(\kappa_1-\kappa)^2}+\frac{6}{25}.$$
	
	\begin{lemma}\label{lema113D}
		We have
		\begin{align*}
		a_{ii}&= \frac{\mu\pi}{(\kappa_1-\kappa)}\left(1+\Big(\frac{\kappa_1+\kappa}{\kappa_1-\kappa}\Big)^{2}K(\kappa_1,\kappa)\right)|\ln\varepsilon|+O(1),\quad i=1,2;\nonumber\\
		a_{33}&=\frac{\mu \pi}{24(\kappa_1-\kappa)^2}\frac{1}{\varepsilon}+O(|\ln\varepsilon|),\quad\frac{1}{C}\leq a_{44}\leq \ C;\nonumber\\
		a_{ii}&=\frac{\mu\pi\,}{(\kappa_1-\kappa)^3}K(\kappa_1,\kappa)|\ln\varepsilon|+O(1),\quad i=5,6;\nonumber\\
		a_{ij}&=\frac{\mu\pi(\kappa_1+\kappa)}{(\kappa_1-\kappa)^3}K(\kappa_1,\kappa)|\ln\varepsilon|+O(1),\quad (i,j)=(1,5)~\mbox{and}~(2,6).
		\end{align*}
	\end{lemma}
	
	\begin{proof}
		We follow the technique used in \cite{LXZ}. By means of  \eqref{defaij}, Proposition \ref{propu11}, and \eqref{v1u1}, we have 
		\begin{align*}
		a_{11}&=\int_{\Omega_R} \left(2\mu e({\bf u}_{1}), e({\bf u}_1)\right)\mathrm{d}x+O(1)\\
		&=\int_{\Omega_R} \left(2\mu e({\bf v}_1), e({\bf v}_1)\right)\mathrm{d}x+\int_{\Omega_R} \left(2\mu e({\bf v}_1), e({\bf w}_1)\right)\mathrm{d}x\\
		&\quad+\int_{\Omega_R} \left(2\mu e({\bf w}_1), e({\bf v}_1)\right)\mathrm{d}x+\int_{\Omega_R} \left(2\mu e({\bf w}_1), e({\bf w}_1)\right)\mathrm{d}x+O(1)\\
		&=\int_{\Omega_R} \left(2\mu e({\bf v}_1), e({\bf v}_1)\right)\mathrm{d}x+O(1),
		\end{align*}    
		where ${\bf w}_1={\bf u}_1-{\bf v}_1$. From \eqref{estv112}--\eqref{estv11223}, we find that the biggest term in $e({\bf v}_1)$ is $\partial_{x_3}{\bf v}_1^{(1)}$ and $\partial_{x_3}{\bf v}_1^{(2)}$. Then we have 
		\begin{align*}
		a_{11}&=\mu\int_{\Omega_R} \left(\frac{1}{\delta(x^\prime)}+\frac{6(\kappa_1+\kappa)}{5}\frac{k(x)}{\delta(x^\prime)}\Big(-\frac{5x_1^2}{\delta(x^\prime)}+\frac{1}{\kappa_1-\kappa}\Big)\right)^2\\
		&\quad+\left(\frac{24(\kappa_1+\kappa)}{5}\frac{x_1x_2}{\delta^2(x^\prime)}k(x)\right)^2\mathrm{d}x+O(1)\\  
		&= \frac{\mu\pi}{(\kappa_1-\kappa)}\left(1+\Big(\frac{\kappa_1+\kappa}{\kappa_1-\kappa}\Big)^{2}K(\kappa_1,\kappa)\right)|\ln\varepsilon|+O(1).
		\end{align*}
		The proof of the estimates of $a_{22}$, $a_{33}$, $a_{55}$, $a_{66}$, $a_{15}$, and $a_{26}$ follows from the same argument, so we omit the details here. 
		For the  estimate of $a_{44}$,
		\begin{align*}
		a_{44}\leq C\int_{\Omega}|\nabla {\bf u}_4|^2\mathrm{d}x\leq C\int_{\Omega_R}\frac{|x'|^2}{\delta^2(x')}\, \mathrm{d}x+C\leq C,
		\end{align*}
		and
		\begin{align*}
		a_{44}\geq \int_{\Omega_R\setminus\Omega_{R/2}}|\partial_{x_3}{\bf v}_{4}^{(2)}|^{2}\geq\int_{\Omega_R\setminus\Omega_{R/2}}\frac{|x'|^{2}}{\delta^{2}(x')}\mathrm{d}x\geq \frac{1}{C}.
		\end{align*} 
		Consequently, the lemma is proved.     	
	\end{proof}

	\begin{lemma}\label{lema114563D}
		For other cases, we have 
		\begin{align*}
		|a_{\alpha\beta}|=|a_{\beta\alpha}|\leq C,\quad \alpha,\beta=1,\dots,6.
		\end{align*}
	\end{lemma}     	
	
	\begin{proof}
		By using \eqref{defaij} and Proposition \ref{propu11}, we have 
		\begin{align*}
		a_{12}&=\int_{\Omega_R} \left(2\mu e({\bf u}_{1}), e({\bf u}_{2})\right)\mathrm{d}x+C\\
		&=\int_{\Omega_R} \left(2\mu e({\bf v}_{1}), e({\bf v}_{2})\right)\mathrm{d}x+\int_{\Omega_R} \left(2\mu e({\bf v}_{1}), e({\bf w}_{2})\right)\mathrm{d}x\\
		&\quad+\int_{\Omega_R} \left(2\mu e({\bf w}_{1}), e({\bf v}_{2})\right)\mathrm{d}x+\int_{\Omega_R} \left(2\mu e({\bf w}_{1}), e({\bf w}_{2})\right)\mathrm{d}x+C\\
		&=\int_{\Omega_R} \left(2\mu e({\bf v}_{1}), e({\bf v}_{2})\right)\mathrm{d}x+C.
		\end{align*}
		Recalling the definitions of ${\bf v}_{\alpha}$ in \eqref{v1alpha}, $\alpha=1,2$, a direct calculation yields that the the biggest terms in $(2\mu e({\bf v}_{1}), e({\bf v}_{2}))$ are $\partial_{x_3}{\bf v}_1^{(1)}\cdot\partial_{x_3}{\bf v}_2^{(1)}$ and $\partial_{x_3}{\bf v}_1^{(2)}\cdot\partial_{x_3}{\bf v}_2^{(2)}$. Note that 
		\begin{align*}
		\partial_{x_3}{\bf v}_1^{(1)}\cdot\partial_{x_3}{\bf v}_2^{(1)}=\frac{2k(x)}{\delta(x^\prime)}G(x)\left(\frac{1}{\delta(x^\prime)}+\frac{2k(x)}{\delta(x_{1})}F_1(x)\right),
		\end{align*} 
		which is an odd function with respect to $x_j$, $j=1,2$, where $F_1(x)$ and $G(x)$ are defined in \eqref{defFGH}. Similarly, $\partial_{x_3}{\bf v}_1^{(2)}\cdot\partial_{x_3}{\bf v}_2^{(2)}$ is also an odd function with respect to $x_j$, $j=1,2$. The integral of the rest terms is bounded. Thus, we derive 
		$$|a_{12}|\leq C.$$
		The rest of the terms is proved  similarly, so we omit the details here. 
	\end{proof}

	\begin{lemma}\label{lemaQb1}
		Assume ${\boldsymbol{\varphi}}\in{\bf\Phi}_1$, where ${\bf\Phi}_1$ is defined in \eqref{defphi1}. If $l_1=0$, then we have
		\begin{align*}
		Q_1[\boldsymbol\varphi]&=\frac{\mu\pi}{(\kappa_1-\kappa)}\left(1+\Big(\frac{\kappa_1+\kappa}{\kappa_1-\kappa}\Big)^{2}K(\kappa_1,\kappa)\right)|\ln\varepsilon|+O(1),\\
		Q_5[\boldsymbol\varphi]&=\frac{\mu\pi(\kappa_1+\kappa)}{(\kappa_1-\kappa)^3}K(\kappa_1,\kappa)|\ln\varepsilon|+O(1),\quad|Q_\beta[\boldsymbol\varphi]|\leq C,\quad \beta=2,3,4,6.
		\end{align*}
		If $l_1\geq1$, then we have
		\begin{align*}
		|Q_\beta[\boldsymbol\varphi]|&\leq C,\quad \beta=1,2,4,5,6,
		\end{align*}    	
		and
		\begin{align*}
		|Q_3[\boldsymbol\varphi]|\leq 
		\begin{cases}
		C|\ln\varepsilon|,&\quad l_1=1,\\
		C,&\quad l_1\geq2.
		\end{cases}
		\end{align*}
	\end{lemma}
	
	\begin{proof}
		If $l_1=0$, then by using the definition of $Q_\beta[\boldsymbol\varphi]$ in \eqref{aijbj}, we find that the estimates of $|Q_\beta[\boldsymbol\varphi]|$ are the same as that of $|a_{1\beta}|$. Next we consider the case of $l_1\geq1$.
		
		If $l_1=1$, then by using Propositions \ref{propu11} and  \ref{propu15}, we have 
		\begin{align}\label{Q11} 
		|Q_1[\boldsymbol\varphi]|&=\left|\int_{\Omega} \left(2\mu e({\bf u}_0),e({\bf u}_{1})\right)\mathrm{d}x\right|\leq\left|\int_{\Omega_R} \left(2\mu e({\bf u}_0),e({\bf u}_{1})\right)\mathrm{d}x\right|+C\nonumber\\
		&\leq C\int_{\Omega_R}\frac{1}{\delta(x^\prime)} \left(\frac{1}{\sqrt{\delta(x^\prime)}}+\frac{|x^\prime|}{\delta(x^\prime)}\right)\mathrm{d}x+C\leq C.  
		\end{align} 
		If $l_1\geq 2$, similar to \eqref{Q11}, we have
		\begin{align*}
		Q_1[\boldsymbol\varphi]|&=\left|\int_{\Omega} \left(2\mu e({\bf u}_{0}),e({\bf u}_{1})\right)\mathrm{d}x\right|\leq C\int_{\Omega_R}\frac{\mathrm{d}x}{\delta(x^\prime)}\ +C\leq C.
		\end{align*}      	
		The estimates of $Q_2[\boldsymbol\varphi]$, $Q_5[\boldsymbol\varphi]$, and $Q_6[\boldsymbol\varphi]$ in the case of $l_1\geq1$ are the same as that of $Q_1[\boldsymbol\varphi]$. We thus omit the details.
		
		In order to estimate $Q_3[\boldsymbol\varphi]$, we first have from Propositions \ref{propu11} and  \ref{propu15} that, for $l_1=1$,
		\begin{align*}
		|Q_3[\boldsymbol\varphi]|&\leq\left|\int_{\Omega_R} \left(2\mu e({\bf u}_0),e({\bf u}_{3})\right)\mathrm{d}x\right|+C\nonumber\\
		&\leq C\int_{|x'|\leq R}\frac{1}{\sqrt{\delta(x^\prime)}} \left(1+\frac{|x^\prime|}{\delta(x^\prime)}\right)\mathrm{d}x+C\leq C|\ln\varepsilon|.
		\end{align*}
		If $l_1\geq2$, using the above process, we have
		$$|Q_3[\boldsymbol\varphi]|\leq C.$$       
		The process of estimating $|Q_4[\boldsymbol\varphi]|$ is easy by using Propositions \ref{propu11} and  \ref{propu15}, we omit the proof here. We thus prove the lemma.
	\end{proof}   
	
	\begin{lemma}\label{lemaQb2}
		Assume ${\boldsymbol{\varphi}}\in{\bf\Phi}_2$, where ${\bf\Phi}_2$ is defined in \eqref{defphi2}. The following assertions hold.
		
		(a) If $l_2=0$, then
		\begin{align*}
		Q_3[\boldsymbol\varphi]=\frac{\mu \pi}{24(\kappa_1-\kappa)^2}\frac{1}{\varepsilon}+O(|\ln\varepsilon|);\quad
		|Q_\beta[\boldsymbol\varphi]|\leq C,\quad \beta=1,2,4,5,6.
		\end{align*}
		
		(b) If $l_2\geq 1$, then 
		\begin{align*}
		|Q_\beta[\boldsymbol\varphi]|\leq
		\begin{cases}
		C|\ln\varepsilon|,&l_2=1,\\
		C,& l_2\geq 2,
		\end{cases} \quad \beta=1,5;\quad
		|Q_\beta[\boldsymbol\varphi]|\leq
		C,\quad l_2\geq 1,\quad \beta=2,3,4,6.
		\end{align*}    	
	\end{lemma}
	
	\begin{proof}
		The proof is similar to Lemma \ref{lemaQb1}. If $l_2=0$, then the estimates of $|Q_\beta[\boldsymbol\varphi]|$ is simiar to that of $|a_{3\beta}|$. For $l_2\geq 1$, we only consider the case of $|Q_1[\boldsymbol\varphi]|$, since other cases are similar. If $l_2=1$, then combining Propositions \ref{propu11} and \ref{propu16}, we have
		\begin{align*}
		|Q_1[\boldsymbol\varphi]|&=\left|\int_{\Omega} \left(2\mu e({\bf u}_0),e({\bf u}_{1})\right)\mathrm{d}x\right|\leq\left|\int_{\Omega_R} \left(2\mu e({\bf u}_0),e({\bf u}_{1})\right)\mathrm{d}x\right|+C\nonumber\\
		&\leq C\int_{\Omega_R}\frac{1}{\delta^{2}(x')} \mathrm{d}x+C\leq C|\ln\varepsilon|.  
		\end{align*}
		If $l_2= 2$, 
		\begin{align*}
		|Q_1[\boldsymbol\varphi]|&=\left|\int_{\Omega} \left(2\mu e({\bf u}_0),e({\bf u}_{1})\right)\mathrm{d}x\nonumber\right|\leq C\int_{\Omega_R}\frac{1}{\delta(x')}\left(\frac{1}{\sqrt{\delta(x')}}+\frac{|x'|}{\delta(x')}\right)\mathrm{d}x+C\leq C.
		\end{align*}
		Similarly, 	if $l_2\geq3$, 
		$$|Q_1[\boldsymbol\varphi]|\leq C.$$
		We thus complete the proof of the estimate of $|Q_1[\boldsymbol\varphi]|$.    
	\end{proof}
	
	Similar to Lemma \ref{lemaQb2}, we have the following result. 
	
	\begin{lemma}\label{lemaQb3}
		If ${\boldsymbol{\varphi}}\in{\bf\Phi}_3$, where ${\bf\Phi}_3$ is defined in \eqref{defphi3}, then we have
		\begin{align*}
		|Q_\beta[\boldsymbol\varphi]|&\leq C, \quad \beta=1,\dots,6.
		\end{align*}     	
	\end{lemma}
	
	Now let us prove Proposition \ref{propu17}.
	\begin{proof}[Proof of Proposition \ref{propu17}.]  
		Denote
		$$ X=(C^1,\dots,C^6)^{\mathrm T},\quad P=(Q_{1}[{\boldsymbol\varphi}],\dots,Q_{6}[{\boldsymbol\varphi}]),\quad \mathbb{A}=(a_{\alpha\beta})_{6\times6}.$$
		Denote $$a:=\frac{1}{\kappa_1-\kappa},~ b:=\frac{(\kappa_1+\kappa)^2}{(\kappa_1-\kappa)^3}K(\kappa_1,\kappa).$$
		Then, in view of Lemma  \ref{lema113D}, it follows that    	
		\begin{align}\label{cofa34}
		& a_{11}a_{22}a_{55}a_{66}+a_{15}^2a_{26}^2-	a_{15}a_{22}a_{51}a_{66}-a_{11}a_{26}a_{55}a_{62}\nonumber\\
		&=\frac{(\mu\pi)^4}{(\kappa_1+\kappa)^4}\Big((a+b)^2b^2+b^4-2(a+b)b^3\Big)|\ln\varepsilon|^4+O(|\ln\varepsilon|^3)\nonumber\\
		&=\frac{(\mu\pi)^4}{(\kappa_1+\kappa)^4}a^2b^2|\ln\varepsilon|^4+O(|\ln\varepsilon|^3).
		\end{align}
		Combining with the fact that $a,b\textgreater 0$, we derive 
		\begin{align}\label{deta}
		\det \mathbb{A}&=
		a_{33}a_{44}\Big(a_{11}a_{22}a_{55}a_{66}+a_{15}^2a_{26}^2-	a_{15}a_{22}a_{51}a_{66}-a_{11}a_{26}a_{55}a_{62}\Big)+O\Big(\frac{|\ln\varepsilon|^3}{\varepsilon}\Big)\nonumber\\
		&\ge \frac{|\ln\varepsilon|^4}{C\varepsilon}.
		\end{align}
		Thus, $\mathbb{A}$ is invertible, and 
		\begin{equation*}
		C^\alpha=\frac{\det\mathbb{A}_{\alpha}}{\det\mathbb{A}},
		\end{equation*} 
		where $\mathbb{A}_{\alpha}$ is a matrix after replacing the $\alpha$-th column of $\mathbb{A}$ with $P$. Next we only prove the case of ${\boldsymbol\varphi}\in{\bf\Phi}_1$, since others are similar.    	
		
		(a) If $l_1=0$,  denote by $\mbox{cof}(\mathbb A)_{ij}$  the cofactor of $\mathbb A$. By making use of \eqref{ce} and Cramer's rule,  we have
		\begin{align}\label{estC1}
		C^1=\frac{1}{\det\mathbb{A}}\Big(&\mbox{cof}(\mathbb A)_{11} Q_{1}[{\boldsymbol\varphi}]-\mbox{cof}(\mathbb A)_{21} Q_{2}[{\boldsymbol\varphi}]+\mbox{cof}(\mathbb A)_{31}Q_{3}[{\boldsymbol\varphi}]-\mbox{cof}(\mathbb A)_{41} Q_{4}[{\boldsymbol\varphi}]\nonumber\\
		&+\mbox{cof}(\mathbb A)_{51} Q_{5}[{\boldsymbol\varphi}]-\mbox{cof}(\mathbb A)_{61} Q_{6}[{\boldsymbol\varphi}] \Big).
		\end{align}
		From Lemma \ref{lemaQb1}, \eqref{cofa34}, and \eqref{deta}, it follows that
		\begin{align*}
		\mbox{cof}(\mathbb A)_{11}=a_{22}a_{33}a_{44}a_{55}a_{66}-a_{26}a_{33}a_{44}a_{55}a_{62}+O(|\ln\varepsilon|^2\varepsilon^{-1}),
		\end{align*}
		and thus
		\begin{align}\label{cofa11}
		\frac{\mbox{cof}(\mathbb A)_{11}}{\det\mathbb{A}}=\frac{\kappa_1-\kappa}{\mu\pi}\frac{1}{|\ln\varepsilon|}+O(|\ln\varepsilon|^{-2}).
		\end{align}
		Similarly, 
		\begin{align}\label{cofa21}
		|\mbox{cof}(\mathbb A)_{21}|,|\mbox{cof}(\mathbb A)_{61}|\leq C|\ln\varepsilon|^2\varepsilon^{-1},~|\mbox{cof}(\mathbb A)_{31}|\leq|\ln\varepsilon|^3,~|\mbox{cof}(\mathbb A)_{41}|\leq C|\ln\varepsilon|^3\varepsilon^{-1},
		\end{align}
		and 
		\begin{equation}\label{cofA}
		\frac{\mbox{cof}(\mathbb A)_{51}}{\det\mathbb{A}}=\frac{(\kappa_1-\kappa)(\kappa_1+\kappa)}{\mu\pi}\frac{1}{|\ln\varepsilon|}+O(|\ln\varepsilon|^{-2}).
		\end{equation}
		Substituting \eqref{cofa11}--\eqref{cofA} into \eqref{estC1}, we have 
		\begin{align}\label{estCi}
		C^1=\frac{\kappa_1-\kappa}{\mu\pi}\left(Q_{1}[{\boldsymbol\varphi}]-(\kappa_1+\kappa)Q_{5}[{\boldsymbol\varphi}]\right)\frac{1}{|\ln\varepsilon|}+O(|\ln\varepsilon|^{-1}).
		\end{align}
		By using  Lemma \ref{lemaQb1}, we have
		$$Q_{1}[{\boldsymbol\varphi}]-(\kappa_1+\kappa)Q_{5}[{\boldsymbol\varphi}]=\frac{\mu\pi}{\kappa_1-\kappa}|\ln\varepsilon|+O(1).$$
		This together with \eqref{estCi} implies that 
		$$C^1=1+O(|\ln\varepsilon|^{-1}).$$
		Similary, using Lemmas \ref{lema113D}--\ref{lemaQb1}, and \eqref{deta}, we obtain 
		\begin{align*}
		|C^2|&=\frac{1}{\det\mathbb{A}}\Big(-\mbox{cof}(\mathbb A)_{12} Q_{1}[{\boldsymbol\varphi}]+\mbox{cof}(\mathbb A)_{22} Q_{2}[{\boldsymbol\varphi}]-\mbox{cof}(\mathbb A)_{32}Q_{3}[{\boldsymbol\varphi}]\nonumber\\
		&\qquad\qquad\quad+\mbox{cof}(\mathbb A)_{42} Q_{4}[{\boldsymbol\varphi}]-\mbox{cof}(\mathbb A)_{52} Q_{5}[{\boldsymbol\varphi}]+\mbox{cof}(\mathbb A)_{62} Q_{6}[{\boldsymbol\varphi}] \Big)\nonumber\\
		&\leq \frac{C\varepsilon}{|\ln\varepsilon|^4}\Big(\frac{|\ln\varepsilon|^3}{\varepsilon}+|\ln\varepsilon|^3\Big)\leq \frac{C}{|\ln\varepsilon|}.
		\end{align*}
		The estimates of $|C^\alpha|$, $\alpha=3,\dots,6$, are proved similarly, we thus omit the details.
		
		(b)  For $l_1\geq 1$, we only prove the estimate of $|C^1|$ for instance, since other terms are proved in the same way. It follows from \eqref{estC1}, Cramer's rule,  \eqref{deta}, Lemmas \ref{lema113D}--\ref{lemaQb1} that 
		\begin{align*}
		|C^1|\leq \frac{C\varepsilon}{|\ln\varepsilon|^4}\Big(\frac{|\ln\varepsilon|^3}{\varepsilon}+|\ln\varepsilon|^4\Big)\leq \frac{C}{|\ln\varepsilon|}.
		\end{align*}
		The proof of Proposition \ref{propu17} is finished.
	\end{proof}
	
	\section{Proof of Upper bounds and Lower bounds}\label{sec5}
	
	In this section we prove the upper bounds of $(\nabla{\bf u},p)$ in Subsection \ref{subsec41} and the lower bounds  in Subsection \ref{subsec42}.
	
	\subsection{Upper  bounds: Proof of  Theorem \ref{mainthm0}}\label{subsec41}
	
	In this Subsection, we  complete the proof of  Theorem \ref{mainthm0} by using Propositions \ref{propu11},  \ref{propu15}, and  \ref{propu17}.  
	
	\begin{proof}[Proof of Theorem \ref{mainthm0}.]   
		We only prove the case of ${\boldsymbol\varphi}\in{\bf\Phi}_1$ for instance, since the case  of ${\boldsymbol\varphi}\in{\bf\Phi}_3$ is the same. 
		
		(a1) If $l_1=0$, then we note that ${\bf u}_{1}+{\bf u}_{0}$ takes the same value on the top and bottom boundaries of $\Omega_{2R}$. By using \eqref{divfree1}, and the proof of Propositions \ref{propu11} and  \ref{propu15}, we find that in $\Omega_{R}$,
		$$|\nabla^{k_1}({\bf u}_{1}+{\bf u}_{0})|+|\nabla^{k_2}(p_{1}+p_{0})|\leq C,\quad k_1=1,2,~k_2=0,1.$$
		This together with  \eqref{udecom}  yields 
		$$|\nabla^{k_1}{\bf u}(x)|\leq|C^1-1||\nabla^{k_1}{\bf u}_{1}(x)|+\sum_{\alpha=2}^{6}|C^{\alpha}||\nabla^{k_1}{\bf u}_{\alpha}(x)|+C,$$
		and 
		\begin{equation}\label{formula-p}
		\nabla^{k_2}p(x)=(C^1-1)\nabla^{k_2}p_{1}(x)+\sum_{\alpha=2}^{6}C^{\alpha}\nabla^{k_2}p_{\alpha}(x)+\nabla^{k_2}(p_{1}+p_{0}).
		\end{equation}
		By applying Propositions \ref{propu11},  \ref{propu15}, and  \ref{propu17}, we have
		\begin{align*}
		|\nabla{\bf u}(x)|&\leq \frac{C}{|\ln\varepsilon|\delta(x^\prime)}+C\varepsilon\left(\frac{1}{\delta(x')}+\frac{|x'|}{\delta^2(x')}\right)+C\left(\frac{|x'|}{\delta(x')}+1\right)\\
		&\leq \frac{C(1+|\ln\varepsilon||x'|)}{|\ln\varepsilon|\delta(x')},
		\end{align*}
		and 
		\begin{align*}
		|\nabla^2{\bf u}(x)|\leq \frac{C}{|\ln\varepsilon|\delta^2(x^\prime)}+C\varepsilon\left(\frac{1}{\delta^2(x')}+\frac{|x'|}{\delta^3(x')}\right)+\frac{C}{\delta(x')}\leq \frac{C(1+|\ln\varepsilon||x'|)}{|\ln\varepsilon|\delta^2(x')}.
		\end{align*}
		In order to estimate $|p|$, we first denote 
		\begin{equation*}\label{defq}
		q(x):=\sum_{\alpha=1}^{6}C^\alpha q_\alpha(x)+q_0(x),\quad(q)_{\Omega_R}:=\frac{1}{|\Omega_{R}|}\int_{\Omega_{R}}q(x)\mathrm{d}x,
		\end{equation*}
		where $q_\alpha=p_\alpha-\bar p_\alpha$ and $q_0=p_0-\bar p_0$. Then by using \eqref{formula-p}, Propositions \ref{propu11}, \ref{propu15}, and  \ref{propu17} again, we have
		\begin{align}\label{estpl1=0}
		|p(x)-(q)_{\Omega_R}|&\leq|C^1-1|| p_1(x)-(q_1)_{\Omega_R}|+\sum_{\alpha=2}^{6}|C^\alpha||p_\alpha-(q_\alpha)_{\Omega_R}|+C\nonumber\\
		&\leq \frac{C}{|\ln\varepsilon|\varepsilon^{3/2}},
		\end{align}
		and 
		\begin{align*}
		|\nabla p(x)|\leq|C^1-1||\nabla p_1(x)|+\sum_{\alpha=2}^{6}|C^\alpha||\nabla p_\alpha|+C\leq \frac{C(1+|\ln\varepsilon||x'|)}{|\ln\varepsilon|\delta^2(x^\prime)}.
		\end{align*}
		
		(a2) If $l_1\geq 1$, using \eqref{equ_v4}, Propositions \ref{propu11},  \ref{propu15}, and  \ref{propu17} again, we obtain
		\begin{align*}
		|\nabla{\bf u}(x)|\leq\sum_{\alpha=1}^{6}|C^{\alpha}||\nabla{\bf u}_{\alpha}(x)|+|\nabla{\bf u}_{0}(x)|
		\leq \frac{C(1+|\ln\varepsilon||x'|)}{|\ln\varepsilon|\delta(x')},
		\end{align*}
		\begin{align*}
		|p(x)-(q)_{\Omega_R}|\leq
		\begin{cases}
		\frac{C}{\varepsilon^{3/2}},&l_1=1,\\ \frac{C}{|\ln\varepsilon|\varepsilon^{3/2}},&l_1\geq2,
		\end{cases}
		\end{align*}
		and 
		\begin{align*}
		|\nabla^2{\bf u}(x)|+|\nabla p(x)|\leq 
		\frac{C(1+|\ln\varepsilon|\sqrt{\delta(x')})}{|\ln\varepsilon|\delta^2(x')}.
		\end{align*}
	\end{proof}
	
	\subsection{Lower bounds: Proof of Theorem \ref{main1thm02}}\label{subsec42}
	In this section, we are devoted to the proof of Theorem \ref{main1thm02}. Before proving the main result, we give a lemma whose proof is a slight modification of that in \cite[Lemmas 7.2 and 7.3]{LXZ}, and we omit the details here.
	
	\begin{lemma}\label{lemmatildeQ}
		For $\beta=1,4,5$, as $\varepsilon\to 0$, we have  
		\begin{equation*}
		Q_{1,\beta}[{\boldsymbol\varphi}]\to Q^*_{1,\beta}[{\boldsymbol\varphi}],
		\end{equation*}
		where $Q_{1,\beta}[{\boldsymbol\varphi}]$ and $Q^*_{1,\beta}[{\boldsymbol\varphi}]$ are defined in  \eqref{defQj} and \eqref{defQj*}, respectively.
	\end{lemma}
	
	Using \eqref{aijbj} and the integration by parts, we have 
	\begin{align*}
	a_{\alpha\beta}=\int_{\partial D_{1}}{\boldsymbol\psi}_\beta\cdot\sigma[{\bf u}_\alpha-{\boldsymbol\psi}_\alpha,p_\alpha]\nu=\int_{\partial D}{\boldsymbol\psi}_\alpha\cdot\sigma[{\bf u}_\beta,p_\beta]\nu.
	\end{align*}
	Similarly, 
	\begin{align*}
	a^*_{\alpha\beta}=\int_{\partial D}{\boldsymbol\psi}_\alpha\cdot\sigma[{\bf u}_\beta^{*},p_\beta^{ *}]\nu,
	\end{align*}
	where $a^*_{\alpha\beta}$ is defined in  \eqref{equ_a*}. Following the argument in the proof of \cite[Lemma 7.1]{LXZ}, we get 
	\begin{align*}
	|{\bf u}_4-{\bf u}_4^*|\leq C\varepsilon^{1/2}\quad\mbox{in}~V\setminus\mathcal{C}_{\varepsilon^{1/4}}, 
	\end{align*}
	where $V:=D\setminus\overline{D_{1}\cup D_{1}^0}$ and $$\mathcal{C}_{r}:=\left\{x\in\mathbb R^{3}\big| |x'|<r,~0\leq x_{3}\leq\varepsilon+\kappa_1r^2\right\},\quad r<R.$$
	
	\begin{lemma}\label{convera}
		For $\alpha=1,\dots,6$, as $\varepsilon\to 0$,  we have
		\begin{equation*}
		a_{\alpha 4}\to a^*_{\alpha 4}.
		\end{equation*}
	\end{lemma}
	
	\begin{proof}[Proof of Theorem \ref{main1thm02}.]  
		We have  from the definition of $k(x)$ in \eqref{def_vx} that
		$$k(0,0,\varepsilon/2)=0.$$
		It follows from \eqref{estv1112} and \eqref{divfree1} that
		\begin{equation*}
		\partial_{x_3}{\bf v}_1^{(1)}(0,0,\varepsilon/2)=\frac{1}{\varepsilon},\quad \partial_{x_3}({\bf v}_{0}^1)^{(1)}(0,0,\varepsilon/2)=-\frac{1}{\varepsilon}.
		\end{equation*}
		Moreover,  from Propositions \ref{propu11} and \ref{propu15}, it follows that
		\begin{align}\label{lower1}
		&|\partial_{x_3}{\bf u}_1^{(1)}(0,0,\varepsilon/2)+\partial_{x_3}{\bf u}_0^{(1)}(0,0,\varepsilon/2)|\nonumber\\\
		&\leq |\partial_{x_3}{\bf v}_1^{(1)}(0,0,\varepsilon/2)+\partial_{x_3}({\bf v}_{0}^1)^{(1)}(0,0,\varepsilon/2)|\nonumber\\
		&\quad+|\partial_{x_3}{\bf w}_1^{(1)}(0,0,\varepsilon/2)+\partial_{x_3}{\bf w}_0^{(1)}(0,0,\varepsilon/2)|\leq C.
		\end{align}
		On the other hand, in view of  Propositions \ref{propu11} and \ref{propu17}, \eqref{v1alpha}, \eqref{v2alpha}, \eqref{v4alpha}, and \eqref{v5alpha}, we obtain
		\begin{equation*}
		\left|\sum_{\alpha=2}^{6}C^\alpha\partial_{x_3}{\bf u}_\alpha^{(1)}(0,0,\varepsilon/2)\right|\leq C.
		\end{equation*}
		This, in combination with \eqref{equ_v4} and \eqref{lower1}, leads to
		\begin{align}\label{Dulower}
		|\nabla{\bf u}(0,0,\varepsilon/2)|&=\left|\sum_{\alpha=1}^{6}C^\alpha\nabla{\bf u}_\alpha(0,0,\varepsilon/2)+\nabla{\bf u}_0(0,0,\varepsilon/2)\right|\nonumber\\
		&\geq \left|\sum_{\alpha=1}^{6}C^\alpha\partial_{x_3}{\bf u}_\alpha^{(1)}(0,0,\varepsilon/2)+\partial_{x_3}{\bf u}_0^{(1)}(0,0,\varepsilon/2)\right|\nonumber\\
		&\geq |C^1\partial_{x_3}{\bf u}_1^{(1)}(0,0,\varepsilon/2)+\partial_{x_3}{\bf u}_0^{(1)}(0,0,\varepsilon/2)|-C\nonumber\\
		&\geq |(C^1-1)\partial_{x_3}{\bf u}_1^{(1)}(0,0,\varepsilon/2)|-C\geq \frac{|C^1-1|}{\varepsilon}-C.
		\end{align}
		From \eqref{defQj} and \eqref{ce}, it follows that 
		$$a_{1\beta}(C^1-1)+\sum_{\alpha=2}^{6}a_{\alpha\beta}C^\alpha=Q_{1,\beta}[{\boldsymbol\varphi}],$$
		where $Q_{1,\beta}[{\boldsymbol\varphi}]$ is defined in \eqref{defQj}. 
		By using Cramer's rule,  \eqref{deta}, and \eqref{cofa11}--\eqref{cofA}, we obtain 
		\begin{align}\label{estC11}
		C^1-1&=\frac{1}{\det\mathbb{A}}\Big(\mbox{cof}(\mathbb A)_{11} Q_{1,1}[{\boldsymbol\varphi}]-\mbox{cof}(\mathbb A)_{21} Q_{1,2}[{\boldsymbol\varphi}]+\mbox{cof}(\mathbb A)_{31}Q_{1,3}[{\boldsymbol\varphi}]\nonumber\\
		&\qquad\qquad-\mbox{cof}(\mathbb A)_{41} Q_{1,4}[{\boldsymbol\varphi}]+\mbox{cof}(\mathbb A)_{51} Q_{1,5}[{\boldsymbol\varphi}]-\mbox{cof}(\mathbb A)_{61} Q_{1,6}[{\boldsymbol\varphi}] \Big)\nonumber\\
		&=\frac{\kappa_1-\kappa}{\mu\pi}\left( Q_{1,1}[{\boldsymbol\varphi}]-(\kappa_1+\kappa) Q_{1,5}[{\boldsymbol\varphi}]-\frac{a_{14}-(\kappa_1+\kappa)a_{54}}{a_{44}}Q_{1,4}[{\boldsymbol\varphi}]\right)\frac{1}{|\ln\varepsilon|}\nonumber\\
		&\quad+O(|\ln\varepsilon|^{-2}).
		\end{align}
		If $Q^*_{1,1}[{\boldsymbol\varphi}]-(\kappa_1+\kappa)Q^*_{1,5}[{\boldsymbol\varphi}]-\frac{a^*_{14}-(\kappa_1+\kappa)a^*_{54}}{a^*_{44}}Q^*_{1,4}[{\boldsymbol\varphi}]\neq 0$, then by using Lemmas \ref{lemmatildeQ} and  \ref{convera}, there exists a small enough constant $\varepsilon_0>0$ such that for $0<\varepsilon<\varepsilon_0$,
		\begin{align*}
		&|Q_{1,1}[{\boldsymbol\varphi}]-(\kappa_1+\kappa)Q_{1,5}[{\boldsymbol\varphi}]-\frac{a_{14}-(\kappa_1+\kappa)a_{54}}{a_{44}}Q_{1,4}[{\boldsymbol\varphi}]|\\
		&\geq \frac{1}{2}|Q^*_{1,1}[{\boldsymbol\varphi}]-(\kappa_1+\kappa)Q^*_{1,5}[{\boldsymbol\varphi}]-\frac{a^*_{14}-(\kappa_1+\kappa)a^*_{54}}{a^*_{44}}Q^*_{1,4}[{\boldsymbol\varphi}]|>0.
		\end{align*}
		Thus, from \eqref{estC11}, we have 
		\begin{align*}
		|C^1-1|\geq\frac{|Q^*_{1,1}[{\boldsymbol\varphi}]-(\kappa_1+\kappa)Q^*_{1,5}[{\boldsymbol\varphi}]-\frac{a^*_{14}-(\kappa_1+\kappa)a^*_{54}}{a^*_{44}}Q^*_{1,4}[{\boldsymbol\varphi}]|}{C}\frac{1}{|\ln\varepsilon|}.
		\end{align*}
		In view of \eqref{Dulower}, we have
		\begin{align*}
		|\nabla{\bf u}(0,0,\varepsilon/2)|\geq\frac{1}{\varepsilon|\ln\varepsilon|}.
		\end{align*}
		The proof of Theorem \ref{main1thm02} is complete.
	\end{proof}

	\section{Ellipsoid suspending particle case}\label{secellipsoid}
	
	In this section, we show that the boundary gradient estimates in Theorem \ref{mainthm0} hold also for the ellipsoid inclusion case. We assume that $D_1$ is an ellipsoid, and near the origin, the part of $\partial D_1$ can be represented by 
	$$x_3=\varepsilon+h_1(x'),\quad\mbox{where}~h_1(x')=\kappa_1x_1^2+\kappa_2x_2^2.$$
	We replace $\delta(x')$ by 
	\begin{equation}\label{newd}
	\delta(x^{\prime})=\varepsilon+(\kappa_1-\kappa)x_1^2+(\kappa_2-\kappa)x_2^2.
	\end{equation}
	The proof is similar to the previous ones, we will only list the main differences in the following.
	
	From the argument in Subsection \ref{secmain}, it follows that the key point is to construct auxiliary functions with the same boundary conditions as ${\bf u}_\alpha$ and ${\bf u}_0$ defined in \eqref{equ_v1} and \eqref{equ_v3}, respectively. For this, we seek ${\bf v}_{1}\in C^{2}(\Omega;\mathbb R^3)$, such that ${\bf v}_{1}={\bf u}_{1}={\boldsymbol\psi}_{1}$ on $\partial{D}_{1}$ and ${\bf v}_{1}={\bf u}_{1}=0$ on $\partial{D}$, 
	\begin{align*}
	{\bf v}_{1}=\boldsymbol\psi_{1}\big(k(x)+\frac{1}{2}\big)+\big({\bf E}_1^{(1)},
	{\bf E}_1^{(2)},{\bf E}_1^{(3)}\big)^{\mathrm T}
	\Big(k^2(x)-\frac{1}{4}\Big)\quad \mbox{in}~\Omega_{2R},
	\end{align*}
	where 
	\begin{align*}
	{\bf E}_1^{(1)}&=-\frac{12(\kappa_1^2-\kappa^2)}{(3\kappa_1+2\kappa_2-5\kappa)}\frac{x_1^2}{\delta(x^\prime)}+\frac{3(\kappa_1+\kappa)}{3\kappa_1+2\kappa_2-5\kappa},\\
	{\bf E}_1^{(2)}&=-\frac{12(\kappa_1+\kappa)(\kappa_2-\kappa)}{(3\kappa_1+2\kappa_2-5\kappa)}\frac{x_1x_2}{\delta(x^\prime)},\\
	{\bf E}_1^{(3)}&=-\delta(x^\prime)\partial_{x_1}k(x){\bf E}_1^{(1)}+\delta(x^\prime)\partial_{x_2}k(x){\bf E}_1^{(2)}+(\kappa_1-\kappa)x_1+2(\kappa_1+\kappa)x_1k(x),
	\end{align*}
	$k(x)$ is defined in \eqref{def_vx} with $\delta(x')$ replaced by \eqref{newd}, and $\|{\bf v}_{1}\|_{C^{2}(\Omega\setminus\Omega_{R})}\leq\,C$.
	We choose $\bar{p}_1\in C^{1}(\Omega)$ such that
	\begin{equation*}
	\bar{p}_1=\frac{6(\kappa_1+\kappa)}{(3\kappa_1+2\kappa_2-5\kappa)}\frac{\mu x_{1}}{\delta^2(x^\prime)}+\mu\partial_{x_3} ({\bf v}_1) ^{(3)}\quad\mbox{in}~\Omega_{2R}.
	\end{equation*}
	It is easy to see that if $\kappa_2=\kappa_1$, then ${\bf v}_{1}$ and $\bar{p}_1$ here are the same as that in \eqref{v1alpha} and \eqref{defp113D} in Section \ref{sec2outline2D}, respectively. The construction for the case of $\alpha=2,3,5,6$ is  a slight modification as before. Thus, using the iteration technique  presented in Section \ref{sec2outline2D}, Proposition \ref{propu11} also holds except for $\alpha=4$.
	
	Next we construct ${\bf v}_{4}\in C^{2}(\Omega;\mathbb R^3)$ satisfying,
	\begin{align*}
	{\bf v}_{4}=\boldsymbol\psi_{4}\big(k(x)+\frac{1}{2}\big)+\big(-3x_2,3x_1,H_4\big)^{\mathrm T}
	\Big(k^2(x)-\frac{1}{4}\Big)\quad \mbox{in}~\Omega_{2R},
	\end{align*}
	where $$H_4=-(\kappa_2-\kappa_1)(2k(x)+1)x_1x_2+\delta(x^\prime)\Big(x_2\partial_{x_1}k(x)-x_1\partial_{x_2}k(x)\Big).$$
	We choose $\bar{p}_4=0$. By a direct calculation, we obtain in $\Omega_{2R}$,
	\begin{align*}
	\nabla\cdot{\bf v}_4=0,\quad |\nabla{\bf v}_4|\leq \frac{|x'|}{\delta(x')},\quad |\mu\Delta{\bf v}_4-\nabla\bar{p}_4|\leq \frac{C|x^\prime|}{\delta^2(x^\prime)}.
	\end{align*}
	Then by applying the energy method  and the iteration technique,  instead of the result of $\alpha=4$ in Proposition \ref{propu11}, we have the following result.
	
	\begin{prop}\label{propu13new}
		Let ${\bf u}_{4}\in{C}^{2}(\Omega;\mathbb R^3),~p_{4}\in{C}^{1}(\Omega)$ be the solution to \eqref{equ_v1}. Then we have for $x\in\Omega_{R}$,
		\begin{equation*}
		\|\nabla({\bf u}_{4}-{\bf v}_{4})\|_{L^{\infty}(\Omega_{\delta/2}(x^\prime))}\leq \frac{C}{\sqrt{\delta(x^\prime)}},
		\end{equation*}
		and
		\begin{equation*}
		\|\nabla^2({\bf u}_{4}-{\bf v}_{4})\|_{L^{\infty}(\Omega_{\delta/2}(x^\prime))}+\|q_4-(q_4)_{\Omega_R}\|_{L^{\infty}(\Omega_{\delta/2}(x^\prime))}\leq \frac{C}{\delta^{3/2}(x^\prime)}.
		\end{equation*}
		Consequently, in $\Omega_R$,
		\begin{align*}
		|\nabla {\bf u}_4(x)|\leq \frac{C}{\sqrt{\delta(x^\prime)}},~|p_4(x)-(q_4)_{\Omega_R}|\leq\frac{C}{\varepsilon^{3/2}},~|\nabla p_4(x)|+|\nabla^2{\bf u}_4(x)|\leq \frac{C}{\delta^{3/2}(x^\prime)}.
		\end{align*}
	\end{prop}
	
	Therefore, with the preparations above, by mimicking the process in the previous sections, the boundary gradient estimates in Theorem \ref{mainthm0} hold for the ellipsoid inclusion and the corresponding estimates for $|p|$ can also be obtained. We omit the details here.

	\section{Proof of Theorem \ref{thmhigher}}\label{prfhigher4}
	This section is devoted to the proof of Theorem \ref{thmhigher}. From \eqref{udecom}, one can see that it suffices to estimate $C^\alpha$, $(\nabla {\bf u}_\alpha,p_\alpha)$, and $(\nabla{\bf u}_0,p_0)$, where $\alpha=1,\dots,\frac{d(d+1)}{2}$. To derive the estimates of $(\nabla {\bf u}_\alpha,p_\alpha)$, we construct ${\bf v}_{\alpha}\in C^{2}(\Omega;\mathbb R^d)$, such that ${\bf v}_{\alpha}={\bf u}_{\alpha}={\boldsymbol\psi}_{\alpha}$ on $\partial{D}_{1}$ and ${\bf v}_{\alpha}={\bf u}_{\alpha}=0$ on $\partial{D}$, especially,
	\begin{align*}
	{\bf v}_{\alpha}=\boldsymbol\psi_{\alpha}\big(k(x)+\frac{1}{2}\big)+{\bf E}_\alpha(x)
	\Big(k^2(x)-\frac{1}{4}\Big),\quad\alpha=1,\dots, d-1,\quad \mbox{in}~\Omega_{2R},
	\end{align*}
	and $\|{\bf v}_{\alpha}\|_{C^{2}(\Omega\setminus\Omega_{R})}\leq\,C$, where $k(x)$ is defined in \eqref{def_vx} with $x_d$ in place of $x_3$,
	\begin{align*}
	{\bf E}_\alpha(x)&=F_\alpha(x) e_\alpha+\sum_{i\neq \alpha, d} G_i(x) e_i\\
	&\quad+\big(H_\alpha(x)-\sum_{i\neq \alpha ,d}\delta(x^\prime)\partial_{x_i}k(x)G_i(x)-\delta(x^\prime)\partial_{x_\alpha}k(x)F_\alpha(x)\big) e_d,
	\end{align*}
	here $\delta(x^\prime)$ is defined in \eqref{55}, 
	\begin{align*}
	\begin{split}
	F_\alpha(x)&:=\frac{12}{1-2d}\frac{(\kappa_1+\kappa)x_\alpha^2}{\delta(x^\prime)}-\frac{3}{1-2d}\frac{\kappa_1+\kappa}{\kappa_1-\kappa},\quad
	G_i(x):=\frac{12}{1-2d}\frac{(\kappa_1+\kappa)x_\alpha x_i}{\delta(x^\prime)},\\
	H_\alpha(x)&:=(\kappa_1-\kappa)x_\alpha+2(\kappa_1+\kappa)x_\alpha k(x).
	\end{split}
	\end{align*}    
	We choose $\bar{p}_\alpha\in C^{1}(\Omega)$ such that 
	\begin{equation*}
	\bar{p}_\alpha=\frac{6\mu}{2d-1}\frac{\kappa_1+\kappa}{\kappa_1-\kappa}\frac{ x_{\alpha}}{\delta^2(x^\prime)}+\mu\partial_{x_d} {\bf v}_{\alpha}^{(d)},\quad\alpha=1,\dots,d-1,\quad\mbox{in}~\Omega_{2R},
	\end{equation*}
	and  $\|\bar{p}_\alpha\|_{C^{1}(\Omega\setminus\Omega_{R})}\leq C$.
	
	For $\alpha=d$, we seek 
	\begin{align*}
	{\bf v}_{d}=\boldsymbol\psi_{d}\big(k(x)+\frac{1}{2}\big)+{\bf E}_d(x)
	\Big(k^2(x)-\frac{1}{4}\Big),\quad \mbox{in}~\Omega_{2R},
	\end{align*}
	and $\|{\bf v}_{d}\|_{C^{2}(\Omega\setminus\Omega_{R})}\leq\,C$, where
	\begin{align*}
	{\bf E}_d(x)=\sum_{i=1}^{d-1}\frac{6x_i}{(d-1)\delta(x^\prime)} e_i+\big(-2k(x)+\sum_{i=1}^{d-1}\frac{6x_i}{(d-1)\delta(x^\prime)}\tilde H_i(x) \big)e_d ,
	\end{align*}
	with
	\begin{align*}
	\tilde H_i=(\kappa_1+\kappa)x_i+2(\kappa_1-\kappa)x_i k(x),
	\end{align*}
	and the associated $\bar{p}_d\in C^{1}(\Omega)$ satisfying, in $\Omega_{2R}$,
	\begin{equation*}
	\bar{p}_d=-\frac{3\mu}{(d-1)(\kappa_1-\kappa)\delta^2(x^\prime)}+\mu\partial_{x_d} {\bf v}_d ^{(d)}.
	\end{equation*} 
	For $\alpha=d+1,\dots,\frac{d(d-1)}{2}$, the construction of ${\bf v}_{4}\in C^{2}(\Omega;\mathbb R^d)$  is as follows:
	\begin{align*}
	{\bf v}_{\alpha}=\boldsymbol\psi_{\alpha}\big(k(x)+\frac{1}{2}\big)\quad\mbox{in}~\Omega_{2R},
	\end{align*}
	and here we can directly take $\bar{p}_\alpha=0$.
	
	Finally, for $\alpha=\frac{d(d-1)}{2}+1,\dots,\frac{d(d+1)}{2},$  we define 
	the auxiliary function ${\bf v}_{\alpha}\in C^{2}(\Omega;\mathbb R^d)$ satisfying
	\begin{align*}
	{\bf v}_{\alpha}=\boldsymbol\psi_{\alpha}\big(k(x)+\frac{1}{2}\big)+{\bf E}_\alpha(x)
	\Big(k^2(x)-\frac{1}{4}\Big),\quad \mbox{in}~\Omega_{2R},
	\end{align*}
	and $ \|{\bf v}_{\alpha}\|_{C^{2}(\Omega\setminus\Omega_{R})}\leq\,C$, where 
	\begin{align*}
	{\bf E}_\alpha(x)&=F_\alpha e_\alpha + \sum_{i\neq \alpha,d}\tilde G_i(x) e_i + \big(2k(x)x_1+\tilde H_{\frac{d(d-1)}{2}-1}(x)\Big(F_\alpha(x)+2k(x)x_d\Big)\\
	&\quad-\sum_{i\neq \alpha, d}\delta(x^\prime)\partial_{x_i}k(x)\tilde G_i(x)\big)e_d,
	\end{align*}
	with
	\begin{align*}
	F_\alpha(x)= \frac{12}{2d-1}\frac{x_{\alpha-\frac{d(d-1)}{2}-1}^2}{\delta(x^\prime)}-2k(x)x_3-\frac{3x_3^2}{\delta(x^\prime)}+\frac{3}{(2d-1)(\kappa_1-\kappa)},
	\end{align*}
	and 
	\begin{align*}
	\tilde G_i(x)=\frac{12}{1-2d}\frac{x_ix_{\alpha-\frac{d(d-1)}{2}-1}^2}{\delta(x^\prime)}.
	\end{align*}
	Then we choose $\bar{p}_\alpha\in C^{1}(\Omega)$ in $\Omega_{2R}$ such that
	\begin{equation*}
	\bar{p}_\alpha=\frac{6\mu}{(2d-1)(\kappa_1-\kappa)}\frac{ x_{\alpha-\frac{d(d-1)}{2}-1}}{\delta^2(x^\prime)}+\mu\partial_{x_d} {\bf v}_{\alpha}^{(d)},\quad \mbox{in}~\Omega_{2R}.
	\end{equation*}
	
	Then we can obtain the following result by repeating the argument as in the proof of Proposition \ref{propu11}.
	
	\begin{prop}\label{propud4}
		Let ${\bf u}_{\alpha}\in{C}^{2}(\Omega;\mathbb R^d)$, $p_{\alpha}\in{C}^{1}(\Omega)$ be the solution to \eqref{equ_v1} with $\alpha=1,\dots,\frac{d(d+1)}{2}$. Then in $\Omega_{R}$, 
		\begin{equation*}
		\|\nabla({\bf u}_{\alpha}-{\bf v}_{\alpha})\|_{L^{\infty}(\Omega_{\delta/2}(x'))}\leq 
		\begin{cases}
		C,&\alpha=1,\dots,d-1,\frac{d(d-1)}{2}+2,\dots,\frac{d(d+1)}{2},\\
		\frac{C}{\sqrt{\delta(x^\prime)}},&\alpha=d,\\
		C\sqrt{\delta(x^\prime)},&\alpha=d+1,\dots,\frac{d(d-1)}{2}+1,
		\end{cases}
		\end{equation*}
		and 
		\begin{equation*}
		\|\nabla q_\alpha\|_{L^{\infty}(\Omega_{\delta/2}(x'))}\leq
		\begin{cases}
		\frac{C}{\delta(x^\prime)},&\alpha=1,\dots,d-1,\frac{d(d-1)}{2}+2,\dots,\frac{d(d+1)}{2},\\
		\frac{C}{\delta^{3/2}(x^\prime)},&\alpha=d,\\
		\frac{C}{\sqrt{\delta(x^\prime)}},&\alpha=d+1,\dots,\frac{d(d-1)}{2}+1.
		\end{cases}
		\end{equation*}
		Consequently, in $\Omega_{R}$,
		\begin{align*}
		|\nabla {\bf u}_{\alpha}(x)|\leq 
		\begin{cases}
		\frac{C}{\delta(x^\prime)},&\alpha=1,\dots,d-1,\frac{d(d-1)}{2}+2,\dots,\frac{d(d+1)}{2},\\
		C\left(\frac{1}{\delta(x^\prime)} +\frac{|x^\prime|}{\delta^2(x^\prime)}\right),&\alpha=d,\\
		C\left(\frac{|x^\prime|}{\delta(x^\prime)} +1\right),&\alpha=d+1,\dots,\frac{d(d-1)}{2}+1,
		\end{cases}
		\end{align*}
		and
		$$|p_{\alpha}(x)-(q_{\alpha})_{\Omega_R}|\leq
		\begin{cases}
		\frac{C}{\varepsilon^{3/2}},&\alpha=1,\dots,d-1,\frac{d(d-1)}{2}+2,\dots,\frac{d(d+1)}{2},\\
		\frac{C}{\varepsilon^{2}},&\alpha=d,\\
		\frac{C}{\sqrt{\varepsilon}},&\alpha=d+1,\dots,\frac{d(d-1)}{2}+1,
		\end{cases}$$
		where $(q_{\alpha})_{\Omega_R}$ is defined in \eqref{defqalpha} with $\alpha=1,\dots,\frac{d(d+1)}{2}$. 
	\end{prop}
	
	For the case of $\boldsymbol\varphi=x_1^l e_1$, Proposition \ref{propu15} also holds for $d\geq4$. Now we are ready to complete the proof of Theorem \ref{thmhigher}.
	
	\begin{proof}[Proof of Theorem \ref{thmhigher}]
		Denote
		\begin{equation*}
		q(x):=\sum_{\alpha=1}^{\frac{d(d+1)}{2}}C^\alpha q_\alpha(x)+q_0(x),\quad(q)_{\Omega_R}:=\frac{1}{|\Omega_{R}|}\int_{\Omega_{R}}q(x)\mathrm{d}x,
		\end{equation*}
		where $q_\alpha=p_\alpha-\bar p_\alpha$ and $q_0=p_0-\bar p_0$.
		By  \eqref{udecom}, the boundedness of $C^\alpha$, Propositions \ref{propud4} and \ref{propu15}, we have \eqref{uppD4}.
	\end{proof}

	\noindent{\bf{\large Acknowledgements.}}
	H.G. Li was partially supported by NSF of China (11971061).

\end{document}